\documentclass [11pt,oneside,a4paper,mathscr]{amsart}
\usepackage{geometry,braket}
\newgeometry{margin=1.3in}
\usepackage[utf8]{inputenc}
\usepackage{amsfonts, amsmath, amssymb, amsthm,mathtools, stmaryrd, blkarray,bm,ytableau} 
\usepackage[backend=biber, style=ext-numeric, sorting=nyt,maxbibnames=99]{biblatex}

\renewbibmacro{in:}{}
\DeclareFieldFormat[article]{volume}{\mkbibbold{#1}}

\DeclareDelimFormat[bib,biblist]{nametitledelim}{\addcolon\space}

\DeclareFieldFormat[article,inbook,incollection,inproceedings,patent,thesis,unpublished]{title}{#1\isdot}

\DeclareFieldFormat{date}{\mkbibparens{#1}}
\DeclareFieldFormat[article,periodical,inbook]{date}{#1}
\DeclareFieldFormat[article,periodical,inbook]{title}{\textit{#1}}
\DeclareFieldFormat[article,periodical]{journaltitle}{#1}
\DeclareFieldFormat[inbook]{booktitle}{#1}
\renewbibmacro*{volume+number+eid}{%
	\printfield{volume}%
	\setunit*{\addnbspace}
	\printfield{number}%
	\setunit{\addcomma\space}%
	\printfield{eid}}%
\DeclareFieldFormat[article]{number}{\mkbibparens{#1}}
\DeclareFieldFormat{pages}{#1}

\usepackage{verbatim}
\usepackage[normalem]{ulem}
\usepackage{tikz-cd}
\usepackage{enumitem}
\usepackage{array}

\usepackage{hyperref}

\theoremstyle{definition}
\newtheorem{theorem}{Theorem}[section]

\newtheorem{lemma}[theorem]{Lemma}
\newtheorem{proposition}[theorem]{Proposition}

\numberwithin{equation}{section}

\newtheorem{conjecture}[theorem]{Conjecture}  

\theoremstyle{definition}
\newtheorem{definition}[theorem]{Definition}

\newtheorem{remark}[theorem]{Remark}

\theoremstyle{remark}
\newcommand{\tdiv}{\text{div}}
\newcommand{\B}{Y}
\newcommand{\Bf}{\mathbf{Y}}
\newcommand{\sgn}{\text{sgn}}

\newcommand{\Pic}{\text{Pic}}

\newcommand{\ba}{\mathbf{a}}
\newcommand{\bb}{\mathbf{b}}

\newcommand{\bu}{\mathbf{u}}
\newcommand{\bv}{\mathbf{v}}
\newcommand{\bG}{\mathbf{G}}
\newcommand{\bY}{\mathbf{Y}}

\newcommand{\bx}{\mathbf{x}}

\newcommand{\bz}{\mathbf{z}}
\newcommand{\bw}{\mathbf{w}}

\newcommand{\BA}{{\mathbb{A}}}

\newcommand{\BC}{{\mathbb{C}}}

\newcommand{\BE}{{\mathbb{E}}}
\newcommand{\BF}{{\mathbb{F}}}

\newcommand{\BI}{{\mathbb{I}}}

\newcommand{\BL}{{\mathbb{L}}}

\newcommand{\BN}{{\mathbb{N}}}

\newcommand{\BP}{{\mathbb{P}}}
\newcommand{\BQ}{{\mathbb{Q}}}
\newcommand{\BR}{{\mathbb{R}}}

\newcommand{\BZ}{{\mathbb{Z}}}

\newcommand{\CB}{{\mathfrak B}}

\newcommand{\CD}{{\mathcal D}}

\newcommand{\CF}{{\mathcal F}}

\newcommand{\CO}{{\mathcal O}}

\newcommand{\CU}{{\mathcal U}}

\newcommand{\pt}{{\mathsf{pt}}}
\newcommand{\td}{{\mathrm{td}}}
\newcommand{\bp}{{\mathbf{p}}}
\newcommand{\bc}{{\mathbf{c}}}

\DeclareFontFamily{OT1}{rsfs}{}
\DeclareFontShape{OT1}{rsfs}{n}{it}{<-> rsfs10}{}
\DeclareMathAlphabet{\curly}{OT1}{rsfs}{n}{it}

\newcommand{\Aut}{\operatorname{Aut}}

\newcommand\Id{\operatorname{Id}}

\newcommand{\bt}{\mathbf{t}}
\newcommand{\bmm}{\mathbf{m}}
\newcommand{\bn}{\mathbf{n}}

\newcommand{\QQ}{\mathbb{Q}}

\newcommand{\ZZ}{\mathbb{Z}}

\newcommand{\Hilb}{\mathsf{Hilb}}

\newcommand{\Sym}{\mathrm{Sym}}

\newcommand{\Tr}{\mathrm{Tr}}

\newcommand{\Tot}{\mathrm{Tot}}

\newcommand{\ch}{\mathsf{ch}}

\newcommand{\wt}{\mathsf{wt}}

\newcommand{\eval}[1]{\left. #1 \right\rvert}
\newcommand{\detty}[1]{\begin{Vmatrix}#1\end{Vmatrix}}

\allowdisplaybreaks
\begin{document}
	\baselineskip=14.5pt
	\title{stable pairs on local curves and Bethe roots}

	\author{Maximilian Schimpf}
	
	\address{Universit\"at Heidelberg, Institut f\"ur Mathematik}
	\email{mschimpf@mathi.uni-heidelberg.de}
	\date{\today}
	
	\begin{abstract}
		We give an explicit formula for the descendent stable pair invariants of all (absolute) local curves in terms of certain power series called Bethe roots, which also appear in the physics/representation theory literature. We derive new explicit descriptions for the Bethe roots which are of independent interest. From this we derive rationality, functional equation and a characterization of poles for the full descendent stable pair theory of local curves as conjectured by Pandharipande and Pixton. We also sketch how our methods give a new approach to the spectrum of quantum multiplication on $\Hilb^n(\BC^2)$.
	\end{abstract}
	\maketitle
	
	\setcounter{tocdepth}{1} 
	\tableofcontents
	\section{Introduction}
	\subsection{Context}\label{sect: beginning}
	Let $X$ be a smooth quasi-projective threefold with a torus $T = (\BC^*)^n$ acting on it so that the fixed locus $X^T$ is projective. Let furthermore $\beta\in H_2(X,\ZZ)$ be a curve class.\\
	One of the most commonly used ways of counting curves of class $\beta$ is via \textit{stable pair theory}, which is defined using the moduli space $P_n(X,\beta)$ of compactly supported stable pairs $[\CO_X \to F]$ with $\chi(F) = n \in \BZ$ and $\ch_2(F) = \beta$. Using the localization formula \cite{localization} and \cite{BehrendFantechi} one can define a virtual class
	\[
		[P_n(X,\beta)]^{vir,T}\in H_{2d_\beta}^T (P_n(X,\beta))_{loc}
	\]
	where $d_\beta \coloneqq \int_\beta c_1(X)$. We furthermore denoted by
	\[
		H_*^T(P)_{loc}\subset H_*^T(P)\otimes_{H^{-*}_T(\pt)} \text{Frac}\left(H^{-*}_T(\pt)\right) 
	\]
	the homogeneous localization of $H_*^T(P)$ as a graded $H^{-*}_T(\pt)$-module - one can similarly define $H^*_T(P)_{loc}$ as the homogeneous localization of $H^{*}_T(\pt)$. Recall also that $H^*_T(\pt)=\QQ[t_1,\ldots,t_n]\eqqcolon \QQ[\bt]$ with $\bt = (t_1,\ldots,t_n)$ and $t_i\coloneqq c_1(L_i)\in H^2_T(\pt)$, where $L_i$ is the standard representation of $T$ associated to its $i$-th coordinate. \\
	Furthermore, let $[\CO \to \BF]$ denote the universal stable pair on $P_n(X,\beta) \times X$, and consider the diagram
	\[
	\begin{tikzcd}
		P_n(X,\beta) \times X \ar{r}{\pi_X} \ar{d}{\pi_{P}} & X \\
		P_n(X,\beta)
	\end{tikzcd}
	\]
	The descendent $\ch_z(\gamma)\in H_T^*(X\times P_n(X,\beta))_{loc}\otimes \QQ[[z]]$ of a class $\gamma \in H^{*}_T(X)$ is defined by\footnote{The pushforward in cohomology is defined as the dual of the pullback in homology which exists as $\pi_P$ is flat \cite[Theorem VIII.5.1]{flatpullbackinhomology}.}
	\[ \ch_z(\gamma) \coloneqq \sum_{k\geq 0} z^k \pi_{P \ast}\left( \ch_k( \BF )\cdot\pi_X^{*}(\gamma)\right)  .\]
	Descendent invariants are given by
	\begin{align}\label{locPT}
		\langle \ch_{z_1}(\gamma_1) \ldots \ch_{z_n}(\gamma_n) \rangle_{\beta}^{X,T} \coloneqq \sum_{m\in\ZZ} p^m \int_{[P_m(X,\beta)]^{vir,T}}\ch_{z_1}(\gamma_1) \ldots \ch_{z_n}(\gamma_n)
	\end{align}
	and take values in $\QQ(\bt)((p))[[\bz]]$. We will often omit $T$ if it is clear from context which torus is used. 
	There are also \textit{connected} stable pair invariants
	\begin{align}\label{locPTconn}
		\langle \ch_{z_1}(\gamma_1) \ldots \ch_{z_n}(\gamma_n) \rangle_{\beta}^{X,T, \text{conn}}\in \QQ(\bt)((p))[[\bz]]
	\end{align}
	which are defined so that the formula
	\begin{align}\label{PTconndisconn}
		\langle \ch_{z_1}(\gamma_1) \ldots \ch_{z_n}(\gamma_n) \rangle_{\beta}^{X,T} = \sum_{\substack{P \text{ partition of }\Set{1,\ldots,n}\\ \beta = \sum_{I\in P} \beta_I}} \frac{\sgn(P)}{\#\Aut(P,(\beta_I)_I)} \prod_{I\in P} \langle \prod_{i\in I} \ch_{z_i}(\gamma_i)\rangle_{\beta_I}^{X,T,\text{conn}}
	\end{align}
	which also relates connected and disconnected Gromov-Witten invariants, holds by fiat. The sum runs over unordered set partitions and $\Aut(P,(\beta_I)_I)$ is the group permuting empty members of $P$ with the same $\beta_I$. Furthermore, $\sgn(P)$ is the sign that arises out of changing the order of the $\ch_{z_i}(\gamma_i)$.\\
	Descendent invariants are expected to be highly constrained:
	\begin{conjecture}[\cite{RahulSurvey}]\label{conj: bigconj}
		For any $X$ and $T$ as above:
		\begin{enumerate}
			\item\label{conj: rationality} The $\bz$-coefficients of \eqref{locPT} and \eqref{locPTconn} are Laurent expansions in $p$ of rational functions i.e. elements in $\QQ(\bt,p)$.
			\item\label{conj: powerminusone} Under the variable change $p\mapsto p^{-1}$ these rational functions transform as follows: 
			\[
				\eval{\langle \ch_{z_1}(\gamma_1) \cdots \ch_{z_n}(\gamma_n) \rangle_{\beta}}_{p\mapsto p^{-1}}^{X,T,*} = p^{-d_\beta} \langle \ch_{-z_1}(\gamma_1) \cdots \ch_{-z_n}(\gamma_n) \rangle_{\beta}^{X,T,*}
			\]
			\item\label{conj: poles} The $\bz$-coefficients of \eqref{locPTconn} may have $p$-poles only at $p=0$ and where $-p$ is an $n$-th root of unity for $1\leq n \leq \text{div}(\beta)$. Here, we took $\text{div}(\beta)$ to be the multiplicity of $\beta$ in the integral cone of curves
			\[
				NE(X) \coloneqq \Set{[C_i] | C_i\subset X \text{ subcurve}}\subset H_2(X,\BZ)/torsion.
			\]
		\end{enumerate}
	\end{conjecture}
	\begin{remark}
		\begin{enumerate}
		\item Conjecture \ref{conj: bigconj} was historically a big driving factor for the development of stable pair theory. Indeed, it was noted in \cite{MNOP2} that the corresponding Donaldson-Thomas descendent theory is irrational. Stable pairs are much better behaved - in particular rationality of the generating series was first conjectured in \cite{PT} with first examples being computed in \cite{PT&BPS}. For toric threefolds and certain complete intersections the rationality was proved in \cite{token7, token4}. The $p\mapsto p^{-1}$ symmetry was first formulated in the Calabi-Yau case in \cite{MNOP1} and related to Serre duality in \cite{PT&BPS}. This led to a proof of rationality and symmetry in the Calabi-Yau case in \cite{BridgeRat, TodaRat} based on the Behrend function approach to enumerative geometry. Unfortunately, this approach does not generalize to the non-Calabi-Yau case. Conjecture \ref{conj: bigconj} was first stated in its full generality in \cite{RahulSurvey}.
		\item Our account of Conjecture \ref{conj: bigconj} differs slightly from the one in \cite{RahulSurvey}, where it was stated only for disconnected invariants\footnote{The present reformulation was suggested by Rahul Pandharipande.}. The connected and disconnected versions of \eqref{conj: rationality} and \eqref{conj: powerminusone} are easily seen to be equivalent. However, the disconnected version of \eqref{conj: poles} does not hold as one can choose $X$ and $\beta$ so that $\beta = \beta_1+\beta_2$ for $\beta_1,\beta_2$ disjoint, independent rigid subcurves satisfying $\tdiv(\beta)=\gcd\left(\tdiv(\beta_1),\tdiv(\beta_2)\right)<\tdiv(\beta_1)$. In this case
		\[
			\langle\emptyset\rangle_\beta^{X} = \langle\emptyset\rangle_{\beta_1}^{X} \cdot \langle\emptyset\rangle_{\beta_2}^{X}
		\]
		has too many poles whereas
		\[
			\langle\emptyset\rangle_\beta^{X,\text{conn}} = 0
		\]
		does not. From \eqref{PTconndisconn} and \eqref{conj: poles} one can however deduce a disconnected version of \eqref{conj: poles} in which $\text{div}(\beta)$ is replaced by
		\[d(\beta) \coloneqq \max\Set{m| \beta = m\beta_1+\beta_2\text{ for }\beta_1,\beta_2\in NE(X), \beta_1> 0}.\]
		Note that these problems do not occur if $NE(X)$ is generated by a single curve as in this case $\text{div}(\beta)=d(\beta)$ and the connected and disconnected versions of Conjecture \ref{conj: bigconj} become equivalent.
		\end{enumerate}
	\end{remark}
	A class of examples that in curve counting is seen as particularly indicative of the general case are \textit{local curves}. By local curves we mean threefolds that are total spaces $X=\Tot_C(L_1\oplus L_2)$ with $C$ a smooth projective curve and $L_1,L_2$ line bundles on $C$. These are then equipped with the standard $T=(\BC^*)^2$-action acting on fibers of $X$. We will always identify a curve class $\beta$ on $X$ with its multiplicity $d = \tdiv(\beta) = d(\beta)$. Indeed, it was observed in \cite{GoJohnnyGo} that one can reduce the GW/PT correspondence for all semi-Fano threefolds to the special case of local curves, which had been checked by explicit computation in \cite{HisNameWasBryan}. In \cite{CopyTheJohnny} we further use these methods to deduce the disconnected version of Conjecture \ref{conj: bigconj} in the absence of descendents. If one could extend the methods of \cite{GoJohnnyGo} to all descendents, one could likely prove Conjecture \ref{conj: bigconj} provided that one can explicitly check it for local curves. Indeed, this check has been partially done in \cite{PaPixRat,PaPixStat}, where \eqref{conj: rationality} is shown for all descendents and \eqref{conj: powerminusone} and \eqref{conj: poles} for stationary and even descendents respectively. Their proof relies on a subtle pole cancellation property of the stable pair vertex as well as the observation made in \cite{virasorooftarget} that curve counting invariants of (local) curves can be effectively computed via degeneration, monodromy invariance and localization. However, it does not seem that one can show \eqref{conj: powerminusone} and \eqref{conj: poles} for all descendents using this approach\footnote{For the pole restriction, the problem is that the algorithm given in \cite{virasorooftarget,PaPixRat,PaPixStat} relies on inverting the cap matrix (c.f. \cite[Section 9.1]{PaPixRat}). However, the fact that the entries of a matrix only have certain poles does not imply that the entries of the inverse also only have said poles.}
	In this paper we replace the aforementioned strategy by a different approach similar to \cite{MonavariOG}, which allows us to show:
	\begin{theorem}\label{thm: structurethm}
		Conjecture \ref{conj: bigconj} holds for all descendent invariants on local curves.
	\end{theorem}
	Relative local curves are conspicuously missing here as the fixed locus of the corresponding stable pair moduli space does not admit as nice of a description as in Section \ref{sect: Fixed}. But for absolute local curves, Theorem \ref{thm: main} gives a closed formula for all stable pair invariants, from which Conjecture \ref{conj: bigconj} can be deduced. Before we state that formula, we must first introduce its constituents: The Bethe roots.
	\subsection{Bethe roots}
	A commonly used tool in the theory of integrable systems is the \textit{algebraic Bethe Ansatz}, which goes back to \cite{StoneAge, BetheOG2}. The key observation is that many integrable systems have an associated system of polynomial equations called \textit{Bethe equations} whose solutions allow one to diagonalize the integrals of motion of that system - see \cite{IntroBethe} for an introduction. The connection to enumerative geometry was first observed in \cite{NS1,NS2}. The Bethe equations most relevant to us come from the quantum intermediate long wave system ($\text{ILW}_1$ in the notation of \cite{litvinov,litvinov2}), which are as follows:
	Let $K$ be the field of Puiseux series
	\[
	K = \overline{\QQ(t_1,t_2)}\{\{p\}\} \coloneqq \bigcup_{n\geq 1} \overline{\QQ(t_1,t_2)}((p^{1/n}))
	\] 
	with $\overline{\BQ(t_1,t_2)}$ the algebraic closure of $\BQ(t_1,t_2)$.  In particular, recall that $K$ is algebraically closed \cite[Corollary 13.15]{Eisenbud}. For fixed $d\geq 1$ we call a tuple $\Bf\in K^d$ \textit{admissible} if $\B_{i}\neq 0,t_1+t_2$ for any $i$ and $\B_i-\B_{i'}\neq t_1,t_2,t_1+t_2$ for all $i\neq i'$. We are then interested in certain admissible tuples satisfying
	\begin{align}\label{eqn: Bethe}
		p = F_{i}(\Bf) 
	\end{align}
	for all $i=1,\ldots,d$ where 
	\begin{equation}\label{eqn: Bethe_formula}
		F_{i}(\Bf) \coloneqq \frac{\B_{i}}{t_1+t_2-\B_{i}}\prod_{\substack{i'\neq i\\0\leq a,b,c\leq 1\\(a,b)\neq (0,0)}} \Big((-1)^c(at_1+bt_2)+ \B_{i'}-\B_{i} \Big)^{(-1)^{a+b+c}}.
	\end{equation}
	From now on we will simply refer to \eqref{eqn: Bethe} as \textit{the Bethe equations}. The particular solutions of these equations that are of interest to us are characterized as follows:
	\begin{theorem}\label{thm: Bethechar}
		For any partition $\lambda$ of size $d$ there is a tuple $\Bf^\lambda \coloneqq \left(\B^\lambda_\Box(p)\right)_{\Box\in\lambda}$ of power series $\B^\lambda_\Box(p)\in \BQ(t_1,t_2)[[p]]$ indexed by the boxes in the Young diagram\footnote{c.f. Section \ref{sect: notation} for the relevant notation.} of $\lambda$ which is uniquely determined by any of the following equivalent descriptions:
		\begin{enumerate}
			\item\label{cond: innit} It is the unique admissible solution of \eqref{eqn: Bethe} in $K^d$ so that 
			\begin{equation}\label{eqn: initialshit}
				\B_{(i,j)}^\lambda(p) = -it_1-jt_2+\CO\left(p^{>0}\right).
			\end{equation}
			for any box $(i,j)\in\lambda$.
			\item\label{cond: recursion} Let the sequence $\bv^n = \left(v_\Box^{n}(p)\right)_{\Box\in\lambda}$ of tuples of power series $v_\Box^{n}(p)\in\QQ(t_1,t_2)[[p]]$ be defined by
			\[
			v^{0}_\Box(p) \coloneqq 0
			\]
			for $n=0$ and for $n>0$ we recursively set
			\[
			v^{n}_{\Box}(p) \coloneqq \frac{p}{\widetilde{F}^\lambda_{\Box}(\Bf^{(\lambda)}(\bv^{n-1}))}
			\]
			where
			\[
			\widetilde{F}^\lambda_{\Box}(\Bf) = (-1)^{d-1}\frac{\B_{\Box}^{1-\delta_{\Box,(0,0)}}}{t_1+t_2-\B_{\Box}} \prod_{\substack{\Box\neq \Box'\in\lambda\\0\leq a,b,c\leq 1\\ (a,b)\neq (0,0)\\ \Box'\neq \Box + (-1)^c(a,b) }}\left(at_1+bt_2+(-1)^c (\B_{\Box'}-\B_{\Box} )\right)^{(-1)^{a+b+c}}
			\]
			and $\Bf^{(\lambda)}(\bv) = \left(\B_\Box^{(\lambda)}(\bv)\right)_{\Box\in\lambda}$has entries given by
			\begin{align}\label{eqn: B(k) defn}
				\B^{(\lambda)}_{(i,j)}(\bv) \coloneqq -it_1-jt_2+\sum_{\substack{\lambda/\mu\text{ conn. skew}\\ (i,j) \in \lambda/\mu}} \prod_{\Box\in \lambda/\mu} v_\Box
			\end{align}
			with the sum running over all connected skew partitions contained in $\lambda$.
			We now have \footnote{All factors occuring in $\widetilde{F}^\lambda_{(i,j)}(\Bf^{(\lambda)}(\bv^{n-1}))$ have a non-zero $p^0$-coefficient. It follows from this that $v_\Box^{n}(p) - v_\Box^{n-1}(p) = \mathcal{O}(p^{n})$ and hence $\B^{(\lambda)}_\Box(\bv^{n}) - \B^{(\lambda)}_\Box(\bv^{n-1}) = \mathcal{O}(p^{n})$ for all $n$. Therefore the limit exists.}
			\[
			\B_\Box^\lambda(p) = \lim_{n\to\infty} \B^{(\lambda)}_\Box(\bv^{n}).
			\]
			\item\label{cond: closedform} One has the following closed formula:
			\[
			\B_\Box^\lambda(p) = 
			[\bv^0]\left(\detty{
				\frac{\partial F_\Box(\Bf^{(\lambda)}(\bv))/\partial v_{\Box'}}{F_\Box(\Bf^{(\lambda)}(\bv))}}
			\cdot\prod_{\Box\in\lambda} \frac{1}{1-p\cdot F_\Box(\Bf^{(\lambda)}(\bv))} \right)
			\]
			where $\detty{\ldots}$ denotes the determinant of a matrix and $[\bv^0]$ means taking the coefficient of $\prod_{\Box\in\lambda} v_\Box^0$ in the expression to the right, all of whose $p$-coefficients turn out to be Laurent series in $\bv$. Furthermore, $\Bf^{(\lambda)}(\bv)$ is as in \eqref{eqn: B(k) defn}.
		\end{enumerate}
	\end{theorem}
	\begin{remark}
		\begin{enumerate}
			\item The uniqueness in \eqref{cond: innit} is not immediate and part of the statement. Also note that the Bethe equations are symmetric in the $\B_i$ which allows us to use the boxes of $\lambda$ as indices.
			\item The Bethe roots are usually characterized using \eqref{cond: innit} though our proof of that is new. However, explicit descriptions like \eqref{cond: recursion} and \eqref{cond: closedform} have to our knowledge not appeared in the literature before.
		\end{enumerate}
	\end{remark}From now on we will call the $\Bf^\lambda$ simply \textit{Bethe roots}. One might hope that they are the only solutions of \eqref{eqn: Bethe}. However, it is easy to see that there are more - for example $\Bf = (\B_i)_{i=1}^d$ with $\B_i = (t_1+t_2)\frac{(-1)^{d-1}p}{(-1)^{d-1}p+1}$ for all $i$ is one such. In order to further narrow things down, we call a solution $\Bf = (\B_i)_{i=1}^d$ \textit{fully admissible} if in addition to being admissible we have $\B_i\neq\B_{i'}$ for all $i\neq i'$. Indeed, this gets rid of all unwanted solutions:
	\begin{theorem}[\cite{litvinov2}]\label{conj: allbethes}
		Up to permutations of tuple-entries, the $\Bf^\lambda$ described in Theorem \ref{thm: Bethechar} are the only fully admissible solutions of the Bethe equations \eqref{eqn: Bethe} over $K$.
	\end{theorem}
	This gives us an entirely algebraic characterization of the Bethe roots whereas the descriptions in Theorem \ref{thm: Bethechar} were all somewhat analytic.
	\subsection{Relative stable pairs and Bethe roots} Although our main methods only apply to absolute local curves, we can still gain information about a relative one:\\
	Consider the threefold $X = \BC^2\times \BP^1$ and the smooth divisor $D = \BC^2\times\Set{0,\infty}$ together with the diagonal $T=\left(\BC^*\right)^2$-action on the $\BC^2$-factor. Recall that one can define moduli spaces $P_n(X/D,d)$ of \textit{relative stable pairs} \cite{LiTheDegen} together with evaluation maps
	\[
		\Hilb^{d}(\BC^2)\xleftarrow{ev_0} P_n(X/D,\beta) \xrightarrow{ev_\infty} \Hilb^{d}(\BC^2)
	\]
	to the Hilbert scheme of points on $\BC^2$ and an equivariant virtual class \[ [P_n(X/D,d)]^{vir}\in H^T_*(P_n(X/D,d))_{loc}. \] Using this one can define invariants via\footnote{One has $P_n(X/D,d)=\emptyset$ for $n<d$.}
	\begin{align*}
		&\langle \ch_{z_1}(\gamma_1)\ldots \ch_{z_n}(\gamma_n)\mid \epsilon,\delta\rangle^{X/D,T} \\
		&\coloneqq \sum_{n\geq d} p^{n-d}\int_{[P_n(X/D,d)]^{vir}} \ch_{z_1}(\gamma_1)\ldots \ch_{z_n}(\gamma_n) ev_0^*(\epsilon) ev_\infty^*(\delta) \in \BQ(t_1,t_2)[[p,\bz]]
	\end{align*}
	where $\epsilon,\delta\in H_T^*(\Hilb^d(\BC^2))$. As observed in \cite{QcohHilb,localDTofcurves} these invariants also encode quantum multiplication on $\Hilb^n(\BC^2)$. Now consider the $\BQ(t_1,t_2)((p))$-vector space \[\mathsf{H}\coloneqq H_T^*(\Hilb^d(\BC^2))\otimes_{\BQ[t_1,t_2]} \BQ(t_1,t_2)((p)).\]
	We encode the stationary invariants of $X/D$ in terms of endomorphism-valued power series
	\[
		M(\bz)\in \mathsf{End}\left(\mathsf{H}\right)[[\bz]]
	\]
	for $\bz=(z_1,\dots,z_n)$.
	Indeed, these are defined by taking the functionals of the shape 
	\begin{align*}
		\mathsf{H}\otimes \mathsf{H}\longrightarrow &\BQ(t_1,t_2)((p))[[\bz]]\\
		\gamma\otimes \delta \longmapsto &\langle \ch_{z_1}(\pt)\ldots \ch_{z_n}(\pt)\mid \gamma,\delta\rangle^{X/D,T}
	\end{align*}
	and using the identification \[\mathsf{End}(\mathsf{H}) = \mathsf{H}^\vee\otimes \mathsf{H} = \mathsf{H}^\vee\otimes \mathsf{H}^\vee = \left( \mathsf{H}\otimes \mathsf{H}\right)^\vee\] which comes from equivariant Poincaré duality for $\Hilb^d(\BC^2)$. 
	Recall furthermore that the set of fixed points $\Hilb^d(\BC^2)^T$ is in natural bijection with the set of partitions $\lambda\vdash d$ i.e. of size $d$ and the associated classes $[\lambda]\in \mathsf{H}$ of the fixed points form a basis of $\mathsf{H}$.
	Further denoting 
	\begin{align}\label{eqn: Es}
		\begin{split}
		E(z,\B) &\coloneqq \frac{(1-e^{-t_1 z})(1-e^{-t_2 z})}{t_1 t_2}e^{z \B}\\
		E(z,\Bf) &\coloneqq  \sum_{i=1}^d E(z,\B_i)
		\end{split}
	\end{align}
	we have:
	\begin{theorem}\label{thm: eigenvalues}
		There is a basis $(v_\lambda)_{\lambda}$ of $\mathsf{H}$ so that
		\[
			v_\lambda = [\lambda] + \CO(p)
		\] 
		and 
		\[
			M(\bz) v_\lambda = \prod_{i=1}^nE(z_i,\Bf^\lambda) v_\lambda
		\]
		for all $n$ and $\lambda\vdash d$, where $\Bf^\lambda = (\B_\Box^\lambda(p))_{\Box\in\lambda}$ is the Bethe root associated to $\lambda$.
	\end{theorem}
	Various versions of this had already been shown in \cite{okag,Feigin,litvinov2} (see also \cite{Smirnov}). We will give a new proof by deducing it from Theorem \ref{thm: main}. See \cite{ShamelessPlug} for a generalization of this approach to the enumerative geometry of Nakajima quiver varieties.
	\subsection{The main Theorem}
	From now on we will fix a curve $C$ of genus $g$ and \[\alpha_1,\ldots,\alpha_g,\beta_1,\ldots,\beta_g\in H^1(C,\ZZ)\] a symplectic basis i.e. so that \[\alpha_i\cdot\alpha_j= \beta_i\cdot\beta_j = 0\text{ and }\alpha_i\cdot\beta_j = \delta_{i,j}\cdot \pt\] for any $1\leq i,j\leq g$ and $\pt\in H^2(C,\ZZ)$ the point class. As a result we get a basis $\mathcal{B}= \Set{ 1,\alpha_1,\ldots,\alpha_g,\beta_1,\ldots,\beta_g,\pt}$ of $H^*(C)$.
	Let further $L_1,L_2$ be line bundles of degrees $l_1,l_2$ respectively and $X=\Tot_C(L_1\oplus L_2)$ the associated local curve.\\
	For fixed degree $d\geq 0$, free variables $\bY = (Y_i)_{i=1}^d$ and classes $\gamma_i \in \mathcal{B}$ we define the formal bracket
	\[
		\langle \ch_{z_1}(\gamma_1) \ldots \ch_{z_n}(\gamma_n) \rangle_{d}^{X,\text{form}}\in \BQ(t_1,t_2,\Bf)[[\bz]]
	\]
	to be the unique super-commutative expression that vanishes in case 
	\[
		\left|\Set{i| \gamma_i = \alpha_l}\right|\neq \left|\Set{i| \gamma_i = \beta_l}\right|
	\]
	for some $l=1,\ldots,g$ and is otherwise given by
	\begin{align}\label{eqn: Plambdamaindef}
		\begin{split}
			&\left\langle \prod_{i=1}^a\ch_{x_i}(1) \cdot\prod_{i=1}^b\ch_{y_i}(\pt)\cdot \prod_{l=1}^g \Big(\ch_{z_1^l}(\alpha_l)\ch_{w_1^l}(\beta_l)\ldots\ch_{z_{c_l}^l}(\alpha_l)\ch_{w_{c_l}^l}(\beta_l)\Big) \right\rangle_{d}^{X,\text{form}} \\
			&= \prod_{i=1}^a x_i\cdot\sum_{\substack{\coprod_{i=-1}^g S_i = \Set{1,\ldots,a}\\ \Box_i\in\lambda\text{ for }i\in S_{-1}}}\prod_{i\in S_-1}\nabla_i^{\Bf} \prod_{i\in S_{-1}} E(x_i,\B_{\Box_i}) \\
			&\cdot \prod_{i\in S_0} \Big(l_1\mathfrak{B}(x_i t_1)+l_2\mathfrak{B}(x_i t_2)\Big)E(x_i,\Bf)
			\cdot\prod_{i=1}^b E(y_i,\Bf)\\
			&\cdot \prod_{i=1}^g  \left( \bz^i,\bw^i;\bx_{S_i}\mid \Bf\right)_{M(\Bf)^{-1}} 
			\cdot A(\Bf)^{g-1}\cdot B_1(\Bf)^{l_1} \cdot B_2(\Bf)^{l_2}.
		\end{split}
	\end{align}
	Here, we denoted
	\[
		M(\Bf) \coloneqq \left(\frac{\partial F_j(\Bf)/\partial \B_i}{F_j(\Bf)}\right)_{i,j}
	\]
	where the $F_i(\Bf)$ are as in \eqref{eqn: Bethe_formula} and 
	\[
		\nabla_i^\Bf \coloneqq \sum_{i'}\left(M(\Bf)^{-1}\right)_{i,i'}\frac{\partial}{\partial \B_{i'}}
	\]
	as well as
	\begin{align}\label{eqn: ABB}
		\begin{split}
			A(\Bf) &\coloneqq \prod_{i=1}^d \left(\B_i \cdot (t_1+t_2-\B_i)\right) \cdot \prod_{\substack{1\leq i,j\leq d\\0\leq a,b\leq 1\\ (a,b,i-j)\neq (0,0,0)}} (at_1+bt_2 + \B_i-\B_j)^{(-1)^{a+b+1}}\cdot\detty{M(\Bf)}\\
			B_1(\Bf) &\coloneqq \prod_{i=1}^d (t_1+t_2-\B_i)^{-1}\cdot \prod_{i,j=1}^d\prod_{b=0}^1\left(t_1+bt_2+\B_i-\B_j\right)^{(-1)^{b+1}}\\
			B_2(\Bf) &\coloneqq \prod_{i=1}^d (t_1+t_2-\B_i)^{-1}\cdot \prod_{i,j=1}^d\prod_{a=0}^1\left(at_1+t_2+\B_i-\B_j\right)^{(-1)^{a+1}}.
		\end{split}
	\end{align} 
	We also wrote 
	\[
	\mathfrak{B}(t) \coloneqq \frac{1}{e^t-1}-\frac{1}{t} = -\frac{1}{2} + \sum_{i\geq 1} B_{2i}\frac{ t^{2i-1}}{(2i)!}
	\]
	with $B_n$ the $n$-th Bernoulli number and
	\[
	\left( \bz,\bw;\bx\mid \Bf\right)_N \coloneqq (-1)^n \sum_{\substack{\ba = (a_i)_{i=1}^m,\\\bb = (b_i)_{i=1}^m,\\\bc = (c_i)_{i=1}^n}} \detty{N_{\ba\sqcup\bc;\bb\sqcup\bc}} \cdot\prod_{i=1}^m z_i w_i E(z_i,\Bf_{a_i}) E(w_i,\Bf_{b_i}) \cdot\prod_{i=1}^n x_i E(x_i,\Bf_{c_i})
	\]
	for any $d\times d$-matrix $N$, vectors $\bz,\bw$ of length $m$ and $\bx$ of length $n$. The $E(z,\Bf)$ are as in \eqref{eqn: Es}. Note in particular that $\langle \dots\rangle^{X,\text{form}}_d$ is symmetric in the $\B_i$.\\
	Our main Theorem now specifies the relationship of the formal bracket with the actual stable pair invariant:
	\begin{theorem}\label{thm: main}
		One can evaluate
		\[
			\langle \ch_{z_1}(\gamma_1) \ldots \ch_{z_n}(\gamma_n) \rangle_{d}^{X,\text{form}}
		\]
		at $\Bf=\Bf^\lambda$ for any Bethe root $\Bf^\lambda$ and we have
		\begin{align}\label{eqn: main fucker}
			\langle \ch_{z_1}(\gamma_1) \ldots \ch_{z_n}(\gamma_n) \rangle_{d}^X = p^{d(1-g)}\sum_{\lambda\vdash d} \langle \ch_{z_1}(\gamma_1) \ldots \ch_{z_n}(\gamma_n) \rangle_{d}^{X,\text{form}}\mid_{\Bf=\Bf^\lambda}.
		\end{align}
		\end{theorem}
	\begin{remark}
		\begin{enumerate}
			\item In \cite{MonavariOG} the case of no insertions and $t_1=-t_2$ was considered. In particular, one should be able to obtain \cite[Theorem 1.3]{MonavariOG} as a special case of Theorem \ref{thm: main}, however one can show that 
			\[
				\eval{\B_{(i,j)}^\lambda(p)}_{t_1=-t_2} = t_1(i-j)
			\]
			and hence setting $t_1=-t_2$ makes some  numerators and denominators in \eqref{eqn: ABB} vanish. This makes it difficult to compute the limit $t_1\to-t_2$.
			\item In case of no insertions, Theorem \ref{thm: main} also holds for an appropriate generating series of Gromov-Witten invariants since Gromov-Witten theory and stable pair theory agree by \cite[Theorem 9.3]{HisNameWasBryan}. This seems to suggest a connection between Bethe roots and Hodge integrals.
		\end{enumerate}
	\end{remark}
	The upshot of Theorem \ref{thm: main} is the following slogan:
	\begin{center}
		\textit{The structure of stable pair invariants\\ is induced by the structure of the Bethe roots!}
	\end{center}
	Indeed, this is how we will prove Theorem \ref{thm: structurethm}. Part \eqref{conj: rationality} will turn out to be a consequence of Theorem \ref{conj: allbethes}, part \eqref{conj: powerminusone} comes from the invariance of the Bethe equations under the involution \[t_i\mapsto t_i,p\mapsto p^{-1},\B_i\mapsto t_1+t_2-\B_i\] and part \eqref{conj: poles} is connected to the fact that the $\B_\Box^\lambda(p)$ are convergent power series that can be locally analytically continued to all $p$ in \[\BC\setminus\Set{\zeta|(-\zeta)^n=1\text{ for some } 1\leq n\leq d }.\]
	\subsection{Content}
	Our proof of Theorem \ref{thm: main} follows the approach of \cite{MonavariOG}, where the double nested Hilbert scheme was first used to study stable pair invariants. In Section \ref{sect: DNHS} we describe the irreducible components of the double nested Hilbert scheme and show that they are all of expected dimension. In Section \ref{sect: Fixed} we recall the relation between the double nested Hilbert scheme and the fixed locus of the moduli of stable pairs and spell out the implications of Section \ref{sect: DNHS} for stable pair invariants. The proofs of Theorems \ref{thm: structurethm}\eqref{conj: rationality}, \ref{thm: structurethm}\eqref{conj: powerminusone}, \ref{thm: eigenvalues} and \ref{thm: main} are carried out in Section \ref{sect: main}. In Section \ref{sect: Bethe} we prove a multivariate version of Theorem \ref{thm: Bethechar} which we then use to show \ref{thm: structurethm}\eqref{conj: poles}. Finally, Section \ref{sect: combi} contains various combinatorial Lemmas which are used in the previous two sections.
	\subsection{Related work}
	During the writing of this paper \cite{DNHSnew} appeared, where the authors study the geometry of the double nested Hilbert scheme for its own sake independently of this paper. In particular, they discuss generalizations of the results of Section \ref{sect: DNHS}.
	\subsection{Acknowledgments}
	First, I would like to thank my advisor Georg Oberdieck for his input, patience and multiple proofreadings of this paper. I also thank Rahul Pandharipande for multiple proofreadings and the invitation to a prolonged research visit during which I was allowed to give a talk on this project. I further thank Alexey Litvinov, Alexei Oblomkov and Andrey Smirnov for answering questions and very helpful discussions. The author was supported by the starting grant ’Correspondences in enumerative geometry: Hilbert schemes, K3 surfaces and modular forms’, No 101041491 of the European Research Council.
	\section{Double nested Hilbert schemes and their irreducible components}\label{sect: DNHS}
	In this section we recall the double nested Hilbert scheme of a smooth projective curve and give a description of its irreducible components, which all turn out to have the same dimension (c.f. Proposition \ref{prop: irrcomps}). This fact is what enabled all calculations in this paper. In Section \ref{sect: notation} we recall some combinatorial notation necessary for stating this result.
	\subsection{Notation}\label{sect: notation}
	By a \textit{partition} $\lambda$ of \textit{size} $d\geq 0$ (or $\lambda\vdash d$ for short) we mean a finite sequence of positive integers $\lambda_0\geq \lambda_1\geq \lambda_2\geq \ldots \geq \lambda_{n-1}> 0$ so that 
	\[
	|\lambda| \coloneqq \sum_{i=0}^{n-1}\lambda_i = d.
	\]
	We call the integer $l(\lambda)\coloneqq n$ the \textit{length} of $\lambda$.
	Furthermore, we will always identify $\lambda$ with its \textit{Young diagram} which is the set of all $(i,j)\in\BN_0^2$ so that $0\leq j <\lambda_i$\footnote{Note that the top left box is denoted $(0,0)$ and not $(1,1)$ as in most of the combinatorics literature.}. For any box $(i,j)\in\lambda$ it is often convenient to write $\Box=(i,j)$ when the coordinates $i$ and $j$ are not used.
	\begin{figure}[h]
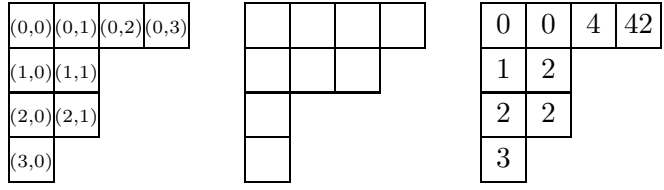

	\begin{ytableau}
		\none &\text{{\tiny (0,0)}} &\text{{\tiny (0,1)}} &\text{{\tiny (0,2)}} &\text{{\tiny (0,3)}}\\
		\none &\text{{\tiny (1,0)}} &\text{{\tiny (1,1)}} \\
		\none &\text{{\tiny (2,0)}} &\text{{\tiny (2,1)}} \\
		\none & \text{{\tiny (3,0)}} 
	\end{ytableau}
	\begin{ytableau}
		\none & & & &\\
		\none & & &\\
		\none &  \\
		\none & 
	\end{ytableau}
	\begin{ytableau}
		\none & 0&0 &4 &42\\
		\none & 1&2 \\
		\none & 2 &2 \\
		\none & 3
	\end{ytableau}
	\caption{On the left: The Young diagram of $\lambda = (4,2,2,1)$ with coordinates. In the middle: The Young diagram of $\overline{\lambda} = (4,3,1,1)$. On the right: A reverse plane partition on $\lambda$}
	\end{figure}\\
	We write $\overline{\lambda}$ for the unique partition whose Young diagram is 
	\[
		\Set{(i,j)| (j,i)\in\lambda}
	\]
	i.e. the Young diagram of $\lambda$ flipped along the diagonal. We further denote 
	\[
		n(\lambda) = \sum_{(i,j)\in\lambda} i.
	\]
	Recall that a \textit{skew partition} $\lambda/\mu$ is a pair of partitions $\lambda$ and $\mu$ so that $\mu_i\leq \lambda_i$ for any $i$. This is equivalent to the Young diagram of $\mu$ being contained in the Young diagram of $\lambda$ and we will often identify $\lambda/\mu$ with the complement of Young diagrams $\lambda\setminus \mu$ - in particular we write $\lvert \lambda/\mu\rvert\coloneqq \lvert \lambda\setminus\mu\rvert$. Note here that a subset $S\subset \lambda$ is a skew partition $S=\lambda/\mu$ if and only if $\Box\in S$ and $\Box\leq \Box'\in\lambda$ imply $\Box'\in S$. 
	\begin{figure}[h]
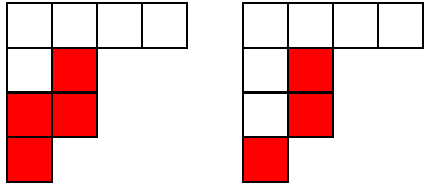

		\begin{ytableau}
			\none & & & &\\
			\none & &*(red) \\
			\none & *(red)&*(red) \\
			\none & *(red)
		\end{ytableau}
		\begin{ytableau}
			\none & & & &\\
			\none & &*(red) \\
			\none & &*(red) \\
			\none & *(red)
		\end{ytableau}
		\caption{Skew partitions $\lambda/\mu_l$ and $\lambda/\mu_r$ respectively with $\lambda=(4,2,2,1)$, $\mu_l=(4,1)$, $\mu_r=(4,1,1)$ and complements in red. The left one is connected and the right one is disconnected.}
	\end{figure}
	Unless stated otherwise, we will from now on require all skew partitions to be \textit{connected} i.e. any two boxes $\Box,\Box'\in\lambda/\mu$ must be connected via a sequence of boxes in $\lambda/\mu$ in which any two consecutive boxes share an edge.\\
	One can equip $\BN_0^2$ with the partial order given by 
	\[
	(i,j)\leq (i',j') \text{ iff } i\leq i' \text{ and } j\leq j'.
	\]
	A tuple of natural numbers $\mathbf{n} = (n_\Box)_{\Box\in \lambda}$ on a partition $\lambda$ is called a \textit{reverse plane partition} if for any $\Box,\Box'\in \lambda$
	\[
	\Box\leq \Box' \text{ implies } n_\Box \leq n_\Box'.
	\]
	In this case we denote
	\[
		\lvert \bn\rvert \coloneqq \sum_{\Box\in\lambda} n_\Box.
	\]
	Observe furthermore that for any tuple of numbers $\bmm = \left(m_{\lambda/\mu}\right)_{\lambda/\mu}$ indexed over the connected skew partitions in $\lambda$ one gets an associated reverse plane partition $\bn = (n_\Box)_{\Box\in\lambda}$ defined by
	\[
		n_{\Box} = \sum_{\Box\in\lambda/\mu} m_{\lambda/\mu}.
	\]
	We will abbreviate this relation as $\lvert\bmm\rvert = \bn$ or $\bmm\vdash \bn$. Further, we will write
	\[
		\lVert\bmm\rVert\coloneqq \lvert\bn\rvert = \sum_{\lambda/\mu} \lvert \lambda/\mu\rvert \cdot m_{\lambda/\mu}.
	\]
	\begin{lemma}\label{lemma: stoopidfuckinlemma}
		Any reverse plane partition $\bn$ is $\lvert\bmm\rvert$ for some $\bmm$ as above. Furthermore, if $n_\Box\leq 1$ for all $\Box\in\lambda$, then $\bmm$ is unique.
	\end{lemma}
	\begin{proof}
		First note that any possibly non-connected skew partition $\lambda/\mu$ is uniquely a disjoint union of connected skew partitions. Indeed, for existence note that the connected components $\lambda/\mu$ are connected skew partitions. For uniqueness note that if $\lambda/\mu = \lambda/\mu_1\sqcup \dots \sqcup\lambda/\mu_n$, then any $\lambda/\mu_i\subset \lambda/\mu$ must be closed under $\leq$ and $\geq$ hence making it a maximal connected subset. Applying this fact to 
		\[
			\lambda/\mu\coloneqq \Set{\Box\in\lambda | n_\Box = 1}
		\]
		we obtain the second claim.\\
		We show the first claim by induction on $\lvert \bn\rvert$. For this let $\lambda/\mu_0$ be a connected component of $ \Set{\Box\in\lambda| n_\Box >  0}$. One easily sees that $\bn'$ defined by 
		\[
			n'_\Box \coloneqq \begin{cases}
			n_\Box-1, \text{ if } \Box\in \lambda/\mu_0,\\
			n_\Box,\text{ else}
		\end{cases}
		\]
		is again a reverse plane partition of smaller size. Hence for any $\bmm'\vdash \bn'$ we get an $\bmm\vdash\bn$ defined by
		\[
			m_{\lambda/\mu} \coloneqq \begin{cases} m'_{\lambda/\mu}+1, \text{ if }\mu = \mu_0\\ m'_{\lambda/\mu},\text{ else.} \end{cases}.
		\]
	\end{proof}
	\subsection{Double nested Hilbert schemes}
	For the rest of this section we fix a partition $\lambda$, a reverse plane partition $\mathbf{n} = (n_\Box)_{\Box\in \lambda}$ and a smooth projective curve $C$ over the complex numbers. We further denote by
	\[
		C^{(n)}\coloneqq \Hilb^n(C) = C^n/\Sym_n
	\]
	the Hilbert scheme of $n$ points on $C$.
	\begin{definition}
		For any tuple of natural numbers $\bmm=(m_i)_{i=1}^l$ we write
		\[
			C^{(\bmm)}\coloneqq \prod_{i=1}^l C^{(m_i)}
		\]
		for the product of Hilbert schemes of $C$. \\
		Furthermore,
		the \textit{double nested Hilbert scheme} associated to $\bn$ is defined as
		\[
		C^{[\bn]} \coloneqq \Set{
			(D_{\Box})_{\Box\in \lambda} | \begin{array}{l} D_{\Box} \subset C\mbox{ divisor of length } n_{\Box} \mbox{ such that}\\ 
			\mbox{for } \Box\leq \Box' \mbox{ we have } D_{\Box}\subset D_{\Box'}\end{array}}\subset C^{(\bn)}.
		\]
	\end{definition}
	\begin{remark}
		More precisely, $C^{[\bn]}$ is the scheme representing the obvious moduli functor. For more details see \cite[Section 2.2]{MonavariOG}.
	\end{remark}
	Given a tuple of nonnegative numbers $\bmm = (m_{\lambda/\mu})_{\lambda/\mu}$ so that $\bmm\vdash \bn$ we get an induced map
	\begin{align*}
		\phi_{\bmm,\bn}\colon C^{(\bmm)}&\longrightarrow C^{[\bn]}\\
		(D_{\lambda/\mu})_{\lambda/\mu} &\longmapsto \left( \sum_{\substack{\Box\in\lambda/\mu}} D_{\lambda/\mu}\right)_{\Box\in \lambda}.
	\end{align*}
	Taking the disjoint union over all such tuples we obtain
	\[
		\phi\colon \coprod_{\bmm\vdash\bn} C^{(\bmm)} \longrightarrow C^{[\bn]}.
	\]
	\begin{proposition}\label{prop: irrcomps}
		The morphism $\phi$ is birational and both sides are pure of dimension 
		\[
			n_{0,0} - \sum_{\substack{(i,j),(k,l)\in\lambda\\0\leq a,b\leq 1\\(k,l) = (i+a,j+b)}} (-1)^{a+b} (n_{k,l}-n_{i,j}).
		\]
		In particular, the fundamental class of $C^{[\mathbf{n}]}$ can be written as
		\[
		\left[C^{[\bn]}\right] = \sum_{\bmm\vdash\bn} \phi_* \left[C^{(\bmm)}\right]
		\]
	\end{proposition}
	\begin{proof}
		To show that $\phi$ is birational it suffices to show that $\phi$ restricts to a bijection $\phi^{-1}(U)\to U$ of dense open subsets and that $C^{[\bn]}$ is generically smooth and pure of the desired dimension. Indeed, this would imply that the domain and codomain of $\phi$ are both pure of the same dimension and by restricting to an appropriate dense open subset of $C^{[\bn]}$, \cite[Proposition 3.17]{Mumford} then implies that $\phi$ restricts to a degree $1$ map to a normal scheme which must be birational.\\
		We first show that $\phi$ is surjective, hence take an arbitrary $\mathbf{D} =(D_\Box)_{\Box\in \lambda}\in C^{[\bn]}$. It suffices to treat the case when $\mathbf{D}$ is supported on a single point $x\in C$ since any $\mathbf{D}$ is a sum of such tuples. In this case it is easy to see that $\mathbf{D}$ is in the image of $\phi_{\bmm,\bn}$ for any $\bmm\vdash\bn$ and by Lemma \ref{lemma: stoopidfuckinlemma} such an $\bmm$ exists.\\
		Furthermore, this process gives a unique preimage if $\mathbf{D}\in S$ for $S\subset C^{[\bn]}$ the open set of tuples $\mathbf{D}$ for which $D_\Box$ is reduced for any $\Box\in\lambda$. Denoting by $S'\subset\coprod_{\bmm\vdash\bn} C^{(\bmm)}$ the dense open consisting of tuples of mutually disjoint reduced divisors it follows that $\phi(S')\subset S$ and by surjectivity of $\phi$ and denseness of $S'$ it follows that $S$ and $\phi^{-1}(S)$ must also be dense. This establishes the generic injectivity.\\
		It remains to check generic smoothness of the double nested Hilbert scheme. For this we consider the closed embedding
		\begin{align*}
			C^{[\bn]} &\longrightarrow C^{(n_{0,0})}\times \prod_{\substack{(i,j)\in \lambda\\ j\geq 1}} C^{(n_{i,j}-n_{i,j-1})}\times \prod_{\substack{(i,j)\in \lambda\\ i\geq 1}} C^{(n_{i,j}-n_{i-1,j})}\eqqcolon X\\
			(D_{i,j})_{(i,j)\in \lambda} &\longmapsto (D_{0,0},(D_{i,j}-D_{i,j-1})_{(i,j)\in\lambda},(D_{i,j}-D_{i-1,j})_{(i,j)\in\lambda}).
		\end{align*}
		This embedding was already considered in \cite[Section 2.4]{MonavariOG} in which it was noted that $C^{[\bn]}$ is cut out by the set of equations
		\[
		D^1_{i-1,j}+D^2_{i,j} = D^2_{i,j-1} + D^1_{i,j}
		\]
		where $(i,j)\in \lambda$ is any box with $i,j\geq 1$ and we denote a point in $X$ by $(D^0,(D^1_\Box)_{\Box\in\lambda},(D^2_\Box)_{\Box\in\lambda})$. Letting $\CU\subset X$ be the open set consisting of tuples of reduced divisors we note that $\CU\cap C^{[\bn]}\subset C^{[\bn]}$ is dense since it contains $S$ defined above. It therefore suffices to show that $\CU\cap C^{[\bn]}$ is smooth. Under the product of the quotient maps $C^n\twoheadrightarrow C^n/S_n = C^{(n)}$ one can pull $\CU$ back to an open set:
		\[
		\widetilde{\CU}\subset C^{n_{0,0}}\times \prod_{\substack{(i,j)\in\lambda\\ j\geq 1}} C^{n_{i,j}-n_{i,j-1}}\times \prod_{\substack{(i,j)\in\lambda\\ i\geq 1}} C^{n_{i,j}-n_{i-1,j}}
		\]
		giving an étale cover $\widetilde{\CU}\to\CU$. This reduces us to showing that the preimage of $\CU\cap C^{[\bn]}$ in $\widetilde{\CU}$, which we denote by $\widetilde{\CU}\cap C^{[\bn]}$, is smooth. Since smoothness is étale local, we may further assume $C = \BA^1$, in which case we denote the coordinates of $\widetilde{\CU}$ by $s_{l},t_{\Box,l},u_{\Box,l}$. The equations cutting out $\widetilde{\CU}\cap C^{[\bn]}$ therefore become:
		\[
		f_{i,j,m} \coloneqq \sum_{l=0}^{n_{i-1,j}-n_{i-1,j-1}} t_{i-1,j,l}^m + \sum_{l=0}^{n_{i,j}-n_{i-1,j}} u_{i,j,l}^m - \sum_{l=0}^{n_{i,j-1}-n_{i-1,j-1}} u_{i,j-1,l}^m - \sum_{l=0}^{n_{i,j}-n_{i,j-1}} t_{i,j,l}^m
		\]
		for $1\leq m \leq n_{i,j}-n_{i-1,j-1}$. We will now show that the Jacobian of these equations has maximal rank by induction on $\lvert\lambda\rvert$. For this we pick point in $\widetilde{\CU}\cap C^{[\bn]}$, a box $(i_0,j_0)\in\lambda$ with $(i_0+1,j_0), (i_0,j_0+1)\not\in\lambda$ and set $\widetilde{\lambda}\coloneqq \lambda\setminus\{(i_0,j_0)\}$. We now need to show that if for a given tuple of complex numbers $(a_{i,j,m})_{i,j,m}$ one has
		\[
		\sum_{i',j',m'} a_{i',j',m'} \cdot\partial_{t_{i,j,l}} f_{i',j',m'} = 0 \text{ and }\sum_{i',j',m'} a_{i',j',m'} \cdot\partial_{u_{i,j,l}} f_{i',j',m'} = 0 
		\]
		at that point for all $i,j,l$, then we must have $a_{i_0,j_0,m}=0$ for all $m$ as the claim for $\widetilde{\lambda}$ gives the rest. Indeed,  looking at partials with respect to $u_{i_0,j_0,l}$ the above equations yield in particular
		\[
		\sum_m a_{i_0,j_0,m} \cdot m u_{i_0,j_0,l}^{m-1} = 0
		\]
		for all $l$, which gives $a_{i_0,j_0,m}=0$ by the distinctness of the $u_{i_0,j_0,l}$ and the invertibility of the Vandermonde matrix.
		Furthermore, this implies that $C^{[\bn]}$ is pure of dimension
		\begin{align*}
			&\dim\left( \CU\right) - \sum_{\substack{(i,j)\in\lambda\\i,j\geq 1}} (n_{i,j}-n_{i-1,j-1}) \\ &= n_{0,0} + \sum_{\substack{(i,j)\in\lambda\\i\geq 1}} (n_{i,j}-n_{i-1,j})+ \sum_{\substack{(i,j)\in\lambda\\j\geq 1}} (n_{i,j}-n_{i,j-1})- \sum_{\substack{(i,j)\in\lambda\\i,j\geq 1}} (n_{i,j}-n_{i-1,j-1})
		\end{align*}
		as desired.
	\end{proof}
	\begin{remark}
		Note that the proof in particular shows that $C^{[\bn]}$ is a local complete intersection. A more thorough study of the geometry of $C^{[\bn]}$ has been undertaken in \cite{DNHSnew}, where Proposition \ref{prop: irrcomps} was independently proved using slightly different methods. In particular, it is shown that $C^{[\bn]}$ is connected and reduced and that $\phi$ is its normalization.
	\end{remark}
	\section{Description of the fixed loci and virtual normal bundle}\label{sect: Fixed}
	In this section we will recall the discussion of \cite[Section 3.3]{MonavariOG} and explore the consequences that Proposition \ref{prop: irrcomps} has for our stable pair calculation. \\
	Resuming the notation in the introduction we let $X = \Tot_C(L_1\oplus L_2)$ be the local curve over $C$ with line bundles $L_i$ of degree $l_i$ and $T=(\BC^*)^2$ the torus acting on $X$.\\
	In order to compute the equivariant stable pair theory of $X$ we must first compute the fixed locus $P_n(X,d)^T$ of the induced $T$-action.\\ Indeed, an element $[\CO_X\to \CF]\in P_n(X,d)^T$ is the same as an equivariant stable pair. By pushing it down to $C$ and decomposing it into its weight spaces, this must be of the shape
	\[
		s=(s_{i,j})\colon\CO_X = \bigoplus_{i,j\geq 0} L_1^{-i}\otimes L_2^{-j} \cdot \bt_1^{-i} \bt_2^{-j} \longrightarrow \bigoplus_{i,j\geq 0} \CF_{i,j}\otimes L_1^{-i}\otimes L_2^{-j} \cdot \bt_1^{-i} \bt_2^{-j}
	\]
	for some coherent $\CF_{i,j}$ on $C$ and morphisms $s_{i,j}\colon \CO_C\to \CF_{i,j}$. Since $\CF$ is of compact support, we must have $\CF_{i,j}=0$ for all but finitely many $i,j$ and the stability is equivalent to each $\CF_{i,j}$ being pure of dimension 1 (hence locally free) and each $s_{i,j}$ having finite cokernel. For any given $i,j$ this forces either $\CF_{i,j} = \CO_C(D_{i,j})$ for some effective divisors $D_{i,j}\subset C$ and $s_{i,j}$ the canonical inclusion or $\CF_{i,j}=0$ and $s_{i,j}=0$. We write $S\subset \BN_0^2$ for the set on which the former happens. The compatibility of $s$ with the multiplication on $\CO_X$ is then equivalent to $S = \lambda$ for some partition $\lambda$ and $D_{\Box}\subset D_{\Box'}$ for any $\Box\leq \Box'\in\lambda$. Therefore $[\CO_X\to \CF]$ corresponds to an element $(D_\Box)_{\Box\in\lambda}\in C^{[\bn]}$ for $n_\Box = \deg D_\Box$. Since this argument can also be performed in flat families, one gets:
	\begin{proposition}[{\cite[Proposition 3.1]{MonavariOG}}]\label{prop: fixedprop}
		The $T$-fixed locus of $P_n(X,d)$ is a disjoint union
		\[
		P_{n}(X,d)^T = \coprod_{\lambda \vdash d} \coprod_{\textbf{n}} C^{[\bn]}
		\]
		where the second disjoint union is over those reverse plane partitions $\textbf{n} = (n_\Box)_{\Box\in\lambda}$ satisfying
		\[
		d(1-g)-n(\lambda)\cdot l_1-n(\bar{\lambda})\cdot l_2 + \lvert\bn\rvert  = n.
		\]
		Furthermore, the K-theory class of the universal stable pair on a component of the fixed locus $C^{[\bn]}$ is given by 
		\[
		\BF = \sum_{(i,j)\in\lambda} \iota_*\CO_{C\times C^{[\bn]}}(\CD_{i,j})\otimes L_1^{-i}\otimes L_2^{-j} \cdot \mathbf{t}_1^{-i} \mathbf{t}_2^{-j} \in K_T^0(X\times C^{[\bn]})
		\]
		where $\CD_{i,j}\subset C\times C^{[\bn]}$ is the universal divisor at the box $(i,j)\in\lambda$ and $\iota\colon C\times C^{[\bn]}\hookrightarrow X\times C^{[\bn]}$ is the inclusion of the zero-section. We wrote $\bt_i\in K_T^0(\pt)$ for the K-theory classes associated to the standard coordinate representations of $T$.\qed
	\end{proposition}
	Using this description and Proposition \ref{prop: irrcomps} one can now express stable pair invariants on $X$ in terms of integrals on symmetric products of C:
	\begin{proposition}\label{prop: localization}
		Given an insertion of the shape $\gamma = \ch_{z_1}(\gamma_1) \ldots \ch_{z_n}(\gamma_n)$ on $X$, the associated stable pair invariant in degree $d$ can be written as
		\begin{equation}\label{eqn: appliedlocalizationformula}
		\langle \gamma \rangle_{d}^{X,T} = p^{d(1-g)}\sum_{\lambda\vdash d} \sum_{\substack{\mathbf{m}=(m_{\lambda/\mu})_{\lambda/\mu}\\ m_{\lambda/\mu}\geq 0}} p^{\lVert\bmm\rVert-n(\lambda)\cdot l_1 - n(\bar{\lambda})\cdot l_2}\int_{C^{(\bmm)}}\frac{\widetilde{\gamma}}{e(N_{\mathbf{m}})}.\end{equation}
		The K-theory class $N_{\mathbf{m}}\in K_{T}^0\left(C^{(\bmm)}\right)$ is given by 
		\begin{align*}
			N_{\mathbf{m}} = &\sum_{\substack{(i,j)\in\lambda\\(i,j)\neq 0}} R\pi_*\left(\CO_{C\times C^{(\bmm)}}(\CD_{i,j})\otimes L_1^{-i}\otimes L_2^{-j}\right) \cdot \mathbf{t}_1^{-i}\mathbf{t}_2^{-j}\\
			&+\sum_{(i,j)\in\lambda} R\pi_*\left(\CO_{C\times C^{(\bmm)}}\left(-\CD_{i,j}\right)\otimes L_1^{i+1}\otimes L_2^{j+1}\right) \cdot \mathbf{t}_1^{i+1}\mathbf{t}_2^{j+1}\\
			&-\sum_{\substack{(i,j),(k,l)\in\lambda\\ 0\leq a,b\leq 1\\(i+a,j+b)\neq (k,l)}} (-1)^{a+b} R\pi_*\left(\CO_{C\times C^{(\bmm)}}\left(\CD_{k,l}-\CD_{i,j}\right)\otimes L_1^{i-k+a}\otimes L_2^{j-l+b}\right)\cdot \bt_1^{i-k+a}\bt_2^{j-l+b}
		\end{align*}
		where 
		\[
		\CD_{\Box} = \sum_{\Box\in\lambda/\mu}\CD_{\lambda/\mu}
		\]
		is a sum over universal divisors $\CD_{\lambda/\mu}\subset C\times C^{(\bmm)}$ and $\pi\colon C\times C^{(\bmm)}\to C^{(\bmm)}$ is the projection onto the second factor. Furthermore, we have
		\[
			\widetilde{\gamma} = \widetilde{\ch_{z_1}(\gamma_1)}\ldots \widetilde{\ch_{z_n}(\gamma_n)}	\in H^*(C^{(\bmm)})		
		\]
		where 
		\begin{equation}\label{eqn: GRRondescendents}
			\widetilde{\ch_{z}(\gamma)} \coloneqq  \frac{(1-e^{-t_1 z})(1-e^{-t_2 z})}{t_1 t_2} \pi_*\left[\ch_z(\widetilde{\BF})  \cdot\bigg(1 + z\pt_C\Big(l_1\CB(zt_1)+l_2\CB(zt_2)\Big)\bigg)\cdot \pi'^*\gamma\right]
		\end{equation}
		with $\pt_C\in H^2(C\times C^{(\bmm)})$ the pullback of the point class along the projection $\pi'\colon C\times C^{(\bmm)}\to C$ to the first factor and
		\begin{equation}\label{eqn: Ftilde}
		\widetilde{\BF} \coloneqq \sum_{(i,j)\in\lambda} \CO_{C\times C^{[\bn]}}(\CD_{i,j})\otimes L_1^{-i}\otimes L_2^{-j} \cdot \mathbf{t}_1^{-i} \mathbf{t}_2^{-j} \in K_T^0(C\times C^{[\bn]}).
		\end{equation}
	\end{proposition}
	\begin{proof}
		Recall that one usually defines stable pair invariants of $X$ in an adhoc way that resembles the virtual localization formula \cite{localization} i.e.
		\begin{equation*}
			\langle \gamma \rangle_{n,d}^{X,T} \coloneqq \int_{[P_n(X,d)^T]^{vir}}\frac{\eval{\gamma}_{P_n(X,d)^T}}{e(N^{vir})}.
		\end{equation*}
		To define everything on the right hand side recall from \cite{HuyThomPOT} that $P_n(X,d)$ has a perfect obstruction theory which is the morphism in $\CD^b(P_n(X,d))$ given by the Atiyah class
		\[
			\BE = R \mathcal{H}om_\pi (\BI,\BI)^\vee_0 [-1] \to \BL_{P_n(X,d)}
		\]
		where the target is the cotangent complex $\BL_{P_n(X,d)}$ of $P_n(X,d)$ with $\BI = [\CO_{X\times P_n(X,d)}\to \CF]$ the universal stable pair on $X\times P_n(X,d)$ and $\pi\colon X\times P_n(X,d)\to P_n(X,d)$ the projection to the second factor. This can be seen to be $T$-equivariant \cite[Example 4.6]{RicolfiAtiyah} so that the invariant part 
		\[
			\eval{\BE^T}_{P_n(X,d)^T} \to \eval{\BL_{P_n(X,d)}^T}_{P_n(X,d)^T} = \BL_{P_n(X,d)^T}
		\]
		of the restriction to the fixed locus is again a perfect obstruction theory and by \cite{BehrendFantechi} this induces a virtual class $[P_n(X,d)^T]^{vir}\in H_*(P_n(X,d)^T)$ whose restriction to any connected component $C^{[\bn]}$ sits in degree $2\cdot \text{rk}\left( \eval{\BE^T}_{C^{[\bn]}}\right)$. Furthermore, the virtual normal bundle is defined as the K-theory class of the non-fixed part:
		\[
			N^{vir}\coloneqq (\eval{\BE^\vee}_{P_n(X,d)^T})^{mov}\in K^0_T(P_n(X,d)^T).
		\]
		Here, $K^0$ denotes the K-theory of locally free sheaves as opposed to $K_0$ the K-theory of coherent sheaves.
		In \cite[Section 4]{MonavariOG} the following identity in $K_0^T(P_n(X,d)^T)$ was shown:
		\begin{align}\label{eqn: Edualdefi}
		\begin{split}
			\eval{\BE^\vee}_{C^{[\bn]}} = &\sum_{(i,j)\in\lambda} R\pi_*\left(\CO_{C\times C^{[\bn]}}(\CD_{i,j})\otimes L_1^{-i}\otimes L_2^{-j}\right) \cdot \mathbf{t}_1^{-i}\mathbf{t}_2^{-j}\\
			&+\sum_{(i,j)\in\lambda} R\pi_*\left(\CO_{C\times C^{[\bn]]}}\left(-\CD_{i,j}\right)\otimes L_1^{i+1}\otimes L_2^{j+1}\right) \cdot \mathbf{t}_1^{i+1}\mathbf{t}_2^{j+1}\\
			&-\sum_{\substack{(i,j),(k,l)\in\lambda\\ 0\leq a,b\leq 1}} (-1)^{a+b} R\pi_*\left(\CO_{C\times C^{[\bn]}}\left(\CD_{k,l}-\CD_{i,j}\right)\otimes L_1^{i-k+a}\otimes L_2^{j-l+b}\right)\cdot \bt_1^{i-k+a}\bt_2^{j-l+b}
		\end{split}
		\end{align}
		where the $D_\Box\subset C\times C^{[\bn]}$ is the universal divisor corresponding to $\Box\in\lambda$ and $\pi\colon C\times C^{[\bn]}\to C^{[\bn]}$ is the projection to the second factor. However, the argument given there also works in $K_T^0(P_n(X,d)^T)$ with pushforwards along equivariant perfect morphisms being defined analogously to \cite[IV.2.12]{SGA6}.
		We now claim that the virtual class $[C^{[\bn]}]^{vir} = [C^{[\bn]}]$ is just the usual fundamental class. Indeed, using Lemma \ref{lemma: vdim=actdim} it suffices to show that $C^{[\bn]}$ is of dimension $\text{rk} \left( \eval{\BE^T}_{C^{[\bn]}}\right)$. Furthermore, one can use Riemann-Roch to calculate this rank:
		\begin{align*}
			\text{rk} \left( \eval{\BE^T}_{C^{[\bn]}}\right) &= \text{rk} \left( R\pi_*\left(\CO_{C\times C^{[\bn]}}(\CD_{0,0})\right)\right)\\ &-\sum_{\substack{(i,j),(k,l)\in\lambda\\ 0\leq a,b\leq 1\\(k,l)=(i+a,j+b)}} (-1)^{a+b} \text{rk}\left( R\pi_*\left(\CO_{C\times C^{[\bn]}}\left(\CD_{k,l}-\CD_{i,j}\right)\right) \right)\\
			&= n_{0,0}+1-g -\sum_{\substack{(i,j),(k,l)\in\lambda\\ 0\leq a,b\leq 1\\(k,l)=(i+a,j+b)}} (-1)^{a+b} \left(n_{k,l}-n_{i,j}+1-g\right)
		\end{align*}
		Using Proposition \ref{prop: irrcomps} it therefore suffices to show that 
		\[
			1 -\sum_{\substack{(i,j),(k,l)\in\lambda\\ 0\leq a,b\leq 1\\(k,l)=(i+a,j+b)}} (-1)^{a+b} \stackrel{!}{=} 0.
		\]
		Indeed, it is easy to see that in the second term the number of pairs $(i,j),(k,l)$ for which $(a,b)=(0,0),(1,0),(0,1)$ and $(1,1)$ is $d, d-l(\overline{\lambda}) = d-\lambda_0, d-l(\lambda)$ and $d-\lambda_0-l(\lambda)+1$ respectively. This establishes the claim. \\
		Because of the projection formula and Proposition \ref{prop: irrcomps} we now only need to show that 
		\begin{equation}\label{eqn: chphicommute}
		\phi_{\bmm,\bn}^*\ch_z(\gamma) = \widetilde{\ch_{z}(\gamma)}
		\end{equation}
		and	$L\phi_{\bmm,\bn}^*N^{vir} = N_{\bmm}$. For the second claim one looks at \eqref{eqn: Edualdefi} and sees that this would easily follow from $L\phi_{\bmm,\bn}^*R\pi_* = R\pi_* L\left(\Id\times\phi_{\bmm,\bn}\right)^*$ with the maps coming from the cartesian diagram:
		\begin{equation}\label{diag: DNHSSymmProdBaseChange}
			\begin{tikzcd}
				C\times C^{(\bmm)}\dar["\pi"] \rar["\Id\times\phi_{\bmm,\bn}"] & C\times C^{[\bn]}\dar["\pi"]\\
				C^{(\bmm)} \rar["\phi_{\bmm,\bn}"] &C^{[\bn]}
			\end{tikzcd}
		\end{equation}
		As $\pi$ is flat, $\phi_{\bmm,\bn}$ and $\pi$ are Tor-independent and therefore the desired commutativity of pushforward and pullback follows from \cite[Proposition IV.3.1.1]{SGA6}. To show \eqref{eqn: chphicommute} we recall that
		\[
			\ch_z(\gamma)= \widetilde{\pi}_*\left(\ch_z(\BF)\cdot(\widetilde{\pi}')^*\gamma\right)
		\]
		where $\widetilde{\pi}\colon X\times C^{[\bn]}\to C^{[\bn]}$ and $\widetilde{\pi}'\colon X\times C^{[\bn]}\to X$ are the projection to the second and first factor respectively. Using equivariant Grothendieck-Riemann-Roch \cite[Corollary after Theorem 1.1]{eqGRR} and $\BF = \iota_*\widetilde{\BF}$ for the zero section $\iota\colon C\times C^{[\bn]}\hookrightarrow X\times C^{[\bn]}$ we obtain
		\begin{align*}
		\ch_z(\BF) & =\iota_* \left(\ch_z(\widetilde{\BF})\cdot \td_{C/ X}^{-1}\right)
		= \iota_* \left(\ch_z(\widetilde{\BF})\cdot\frac{(1-e^{-(t_1+l_1\pt_C) z})(1-e^{-(t_2+l_2\pt_C) z})}{(t_1+l_1\pt_C) (t_2+l_2\pt_C)}\right)\\
		&=\frac{(1-e^{-t_1 z})(1-e^{-t_2 z})}{t_1 t_2} \iota_* \left[\ch_z(\widetilde{\BF})\cdot\bigg(1 + z\pt_C\Big(l_1\CB(zt_1)+l_2\CB(zt_2)\Big)\bigg) \right].
		\end{align*}
		From the projection formula and $\pi = \widetilde{\pi}\circ \iota$ and $\pi' = \widetilde{\pi}'\circ\iota$ it then follows that
		\begin{align*}
			\ch_z(\gamma)&=\frac{(1-e^{-t_1 z})(1-e^{-t_2 z})}{t_1 t_2} \pi_*\Bigg[\ch_z(\widetilde{\BF})\cdot
			 \bigg(1 + z\pt_C\Big(l_1\CB(zt_1)+l_2\CB(zt_2)\Big)\bigg)\cdot\pi'^*\gamma\Bigg]
		\end{align*}
		where pullbacks and pushforwards are along $\pi\colon C\times C^{[\bn]}\to C^{[\bn]}$ and $\pi'\colon C\times C^{[\bn]}\to C$. It now suffices to show that \eqref{diag: DNHSSymmProdBaseChange} satisfies  $\phi_{\bmm,\bn}^*\pi_* = \pi_* \left(\Id\times\phi_{\bmm,\bn}\right)^*$ in cohomology, which is dual to $\pi^*\left(\phi_{\bmm,\bn}\right)_* = \left(\Id\times\phi_{\bmm,\bn}\right)_*\pi^*$ in homology. This however follows from flat base change \cite[Theorem VIII.5.1(2)]{flatpullbackinhomology}, which finishes the proof.
	\end{proof}
	The following Lemma was needed in the above proof and is a well-known piece of folklore for which the author claims no originality. However, due to the apparent lack of a reference and for the sake of completeness we give a full proof.
	\begin{lemma}\label{lemma: vdim=actdim}
		Let $X$ be of finite type over some field $k$ and of pure dimension $d$ together with a perfect obstruction theory $\phi\colon\BE \to \BL_X$ of rank $\text{rk}  \left(\BE\right) = d$. It follows that $\phi$ is a quasi-isomorphism, $X$ is lci and the virtual class must agree with the fundamental class i.e.
		\[
			[X]^{vir} = [X]\in CH_d(X).
		\]
	\end{lemma}
	\begin{proof}
		Since this is a local question we may assume that $\BE = [E^{-1}\to E^0]$ for $E^i$ locally free and that the morphism $\BE\to \BL_X$ results from a commutative diagram
		\begin{center}
			\begin{tikzcd}
				E^{-1}\rar["\phi"]\dar & E^0\dar\\
				I/I^2 \rar["\delta"] & \Omega_M|_X
			\end{tikzcd}
		\end{center}
		where we used $\tau_{\geq -1}\BL_X = [I/I^2\to \Omega_M|_X]$ for $I$ the ideal sheaf cutting out a closed embedding $X\hookrightarrow M$ into $M$ non-singular.
		Recall that $\BE\to \tau_{\geq -1}\BL_X$ is an obstruction theory if and only if 
		\[
			E^{-1}\to E^0\oplus I/I^2\to \Omega_M|_X \to 0
		\]
		is exact. After tensoring with the residue field $k(x)$ for some arbitrary $x\in X$ we therefore obtain
		\[
			\dim \left( I \otimes k(x)\right) = \dim \left( I/I^2 \otimes k(x)\right) \leq \dim(M)-d
		\]
		meaning that $I$ is locally generated by at most (indeed exactly) $\dim(M)-d$ elements which must form a regular sequence. Therefore $X$ is a local complete intersection and $I/I^2$ must be a vector bundle of rank $\dim(M)-d$. It then follows from rank considerations that
		\[
		0\to E^{-1}\to E^0\oplus I/I^2\to \Omega_M|_X \to 0
		\]
		is exact i.e. $\BE\to \tau_{\geq -1}\BL_X=\BL_X$ is a quasi-isomorphism. Finally, \cite[Proposition 5.3]{BehrendFantechi} and the example following it give us $[X]^{vir} = [X]$.
	\end{proof}
	\section{The main computation}\label{sect: main}
	\subsection{Intersection theory of symmetric products of curves}\label{sect: IntTheorySymm}
	We now want to compute the integrals appearing in Proposition \ref{prop: localization}. In order to do that we first need to recall the intersection theory of the symmetric product of a connected smooth projective curve $C$ of genus $g$ as outlined for example in \cite{curvebook}.\\
	For any $m>0$ and fixed $c\in C$ we have an embedding 
	\begin{align*}
		\iota\colon C^{(m-1)} &\hookrightarrow C^{(m)}\\
		D &\longmapsto D+c
	\end{align*}
	of a divisor with cohomology class $u\in H^2(C^{(m)})$. Note that $u$ does not depend on the choice of $c$ as $C$ is connected. More generally, for any $n\geq 0$ the subvariety 
	\begin{align*}
		\iota_{n}\colon C^{(m-n)} &\hookrightarrow C^{(m)}\\
		D &\longmapsto D+m\cdot c
	\end{align*}
	represents $u^n\in H^{2n}(C^{(m)})$, which in particular implies $\int_{C^{(m)}} u^n = \delta_{m,n}$. Furthermore, we have the Abel-Jacobi maps 
	\begin{align*}
		\text{AJ}_m\colon C^{(m)} &\longrightarrow \text{Pic}^0(C)\\ 
		D &\longmapsto \CO_C (D-m\cdot c)
	\end{align*}
	to the Jacobian of $C$, which are compatible with the $\iota_n$. It can be shown that this induces an isomorphism on $H^1$ and so $H^1(C^{(m)}) = H^1(\Pic^0(C)) = H^1(C)$. Recall that $H^1(C)$ has a symplectic base $\alpha_1,\ldots, \alpha_g,\beta_1,\ldots,\beta_g\in H^1(C)$ i.e. so that $\int_{C} \alpha_i \beta_j = \delta_{i,j}$ and $\int_{C} \alpha_i\alpha_j = \int_{C} \beta_i\beta_j = 0$ for all $1\leq i,j\leq g$. By abuse of notation we will also denote the pullback of this basis by $\alpha_i,\beta_i$ in $H^1(\Pic^0(C))$ and in $H^1(C^{(m)})$. Furthermore, we will denote by $\theta\in H^2(\Pic^0(C))$ the theta divisor \cite[Section I.4]{curvebook} as well as its pullback along $\text{AJ}_m$ in $H^2(C^{(m)})$. More explicitly, we have
	\[
	\theta = \sum_{i=1}^g \alpha_i \beta_i.
	\]
	We can now state Poincaré's formula \cite[Section I.5]{curvebook} which says
	\[
	\left(\text{AJ}_m\right)_*(1) = \begin{cases}
		\frac{\theta^{g-m}}{(g-m)!}, \text{ if } m\leq g\\
		0, \text{ otherwise}
	\end{cases}
	\]
	in particular, one can rewrite this without case distinction as 
	\[
	\left(\text{AJ}_m\right)_*(1) = \sum_{\substack{I\subset \{1,\ldots,g\}\\\lvert I \rvert = g-m}} \prod_{i\in I} \alpha_i\beta_i.
	\]
	The following Lemma will be very useful for computations later on:
	\begin{lemma}
		For any $m,n\geq 0$ and $I\subset \{1,\ldots,g\}$ we have
		\begin{align*}
			\int_{C^{(m)}} u^n \prod_{i\in I}\alpha_i\beta_i = \delta_{\lvert I\rvert+n,m}
		\end{align*}
		and for any two distinct subsets $I,J\subset\{1,\ldots,g\}$
		\begin{align*}
			\int_{C^{(m)}} u^n \prod_{i\in I}\alpha_i\prod_{j\in J}\beta_j = 0.
		\end{align*}
	\end{lemma}
	\begin{proof}
		For the first integral note that the second factor is pulled back under the Abel-Jacobi map, which commutes with $\iota_n\colon C^{(m-n)}\subset C^{(m)}$. Using the projection formula we can rewrite the integral as
		\begin{align*}
			\begin{split}
			\int_{C^{(m)}} u^n \prod_{i\in I}\alpha_i\beta_i &= \int_{\Pic^0(C)}\left( \text{AJ}_{m-n}\right)_*(1) \prod_{i\in I}\alpha_i\beta_i \\
			&= \sum_{\substack{J\subset \{1,\ldots,g\}\\\lvert J \rvert = g-m+n}} \int_{\Pic^0(C)} \prod_{j\in J} \alpha_j\beta_j \prod_{i\in I}\alpha_i\beta_i
			\end{split}
		\end{align*}
		where we used Poincaré's formula in the last equality. Now the only summand that can contribute is the one corresponding to $J = I^c$, which occurs in the sum only if $m-n=\lvert I\rvert$, proving the first claim. The second claim follows by a similar argument.
	\end{proof}
	\begin{remark}\label{rem: replace_by_x}
		This means that whenever we are computing an integral on $C^{(m)}$ we can replace the expression $\prod_{i\in I}\alpha_i\beta_i$ by $u^{\lvert I \rvert}$ without changing the value of the integral.
	\end{remark}
	Now denote by $\CD\subset C\times C^{(m)}$ the universal divisor of $C^{(m)}$. As shown in \cite[p.354]{curvebook} one can write it in cohomology as  
	\begin{equation}\label{eqn: Duniv}
		\CD = m\cdot \pt_C + u + \gamma
	\end{equation}
	with $\pt_C$ and $u$ implicitly pulled back from $C$ and $C^{(m)}$ respectively while $\gamma$ is defined as
	\[
	\gamma\coloneqq \sum_{i=1}^g \left(\beta_i\times\alpha_i - \alpha_i\times\beta_i\right)\in H^2(C\times C^{(m)}).
	\]
	It is easy to check that $\gamma^2 = -2\pt_C\cdot\theta$ and $\pt_C\cdot\gamma = \gamma^3 = 0$.
	Denoting for any K-theory class $\CF\in K^0(X)$ on a scheme $X$:
	\[
	c_T(\CF)\coloneqq \sum_{i\geq 0} c_i(\CF) \cdot T^{\text{rk}(\gamma)-i} \in H^*(X)[ T,T^{-1}]
	\]
	we now have:
	\begin{lemma}\label{usefulcohclass}
		For any tuple $\mathbf{m} = (m_i)_{i=1}^n$ of nonnegative integers, $a_1,\ldots,a_n\in \BZ$ and $L$ a line bundle of degree $l$ on $C$:
		\begin{align}\label{eqn: c(Rpi O(D))}
			\begin{split}
			&c_T\left(R\pi_*\left(\CO_{C\times C^{(\bmm)}}\left(\sum_{i=1}^n a_i \CD_i\right) \otimes L\right)\right) \\
			&=\exp\left(-\frac{\sum_{k,l=1}^n a_k a_l \theta_{k,l}}{T+\sum_{i=1}^n a_i u_i}\right)\cdot\left(T+\sum_{i=1}^n a_i u_i\right)^{1-g+l+\sum_{i=1}^n a_i m_i}
			\end{split}
		\end{align}
		where $\pi\colon C\times C^{(\bmm)}\to C^{(\bmm)}$ is the projection to the second factor, $\CD_i\subset C\times C^{(\bmm)}$ is the universal divisor divisor of the $i$-th factor, $\theta_{k,l}$ is defined as
		\[
		\theta_{k,l} = \sum_{i=1}^g \alpha_i^k \times\beta_i^l
		\]
		and $\gamma^k$ is understood to be pulled back from the projection $C^{(\bmm)} = \prod_{i=1}^n C^{(m_i)}\to C^{(m_k)}$ to the $k$-th factor.
	\end{lemma}
	\begin{proof}
		We closely follow the Proof of \cite[Lemma VIII.2.5]{curvebook}. First we write
		\begin{align*}
			\CD = \sum_{i=1}^n a_i \CD_i, \ M = \sum_{i=1}^n a_i m_i, \ u = \sum_{i=1}^n a_i u_i, \ \gamma = \sum_{i=1}^n a_i \gamma_i, \theta = \sum_{i,j=1}^n a_i a_j \theta_{i,j}
		\end{align*}
		and $\CF = R\pi_*\left(\CO_{C\times C^{(\bmm)}}\left(\CD\right) \otimes L\right)$ for short.
		Using Grothendieck-Riemann-Roch we can further compute
		\begin{align*}
			ch(\CF\otimes \CO(-u)) &= \pi_* \left(e^{\CD + l\pt_C-u} \cdot \td_C\right) = \pi_* \left(e^{(M+l)\pt_C+\gamma} \cdot (1+(1-g)\pt_C)\right)\\
			&= \pi_* \left((1+(M+l-\theta)\pt_C + \gamma) \cdot (1+(1-g)\pt_C)\right) =  M+l+1-g-\theta
		\end{align*}
		where we have used $\gamma\cdot \pt_C = 0$, $\gamma^2 = -2 \pt_C \cdot \theta$ and $\gamma^3=0$.
		Recall further that the conversion from chern character to chern class reads
		\[
			\sum_{k\geq 0} c_k(\CF)\cdot t^k = \exp\left(\sum_{k\geq 1}(-1)^{k-1} (k-1)! \ch_k(\CF)\cdot t^k\right)
		\]
		and therefore
		\begin{align*}
			c_T (\CF\otimes \CO(-u)) &= T^{\text{rk} (\CF)}\sum_{k\geq 0} c_k(\CF\otimes \CO(-u)) \cdot T^{-k} \\
			&= T^{M+l+1-g} e^{-\frac{\theta}{T}}
		\end{align*}
		where we have used $\text{rk} (\CF) = M+l+1-g$ which follows from Riemann-Roch.  This finally implies
		\[
		c_T(\CF) = c_{T+u}(\CF\otimes\CO(-u)) = (T+u)^{M+l+1-g} e^{-\frac{\theta}{T+u}}
		\]
		as desired.
	\end{proof}
	\begin{remark}
		Note that one can rewrite \eqref{eqn: c(Rpi O(D))} as
		\[
		\left( T+\sum_i a_i u_i\right)^{1-g+l}\exp\left(-\sum_{k,l=1}^n \theta_{k,l}\frac{\partial f_k(\bu)/\partial u_l}{f_k(\bu)}\right)\prod_{k=1}^n f_k(\bu)^{m_k} 
		\]
		where we treat $u_l$ in
		\[
		f_k(\bu) = \left( T+\sum_i a_i u_i\right)^{a_k}.
		\]
		as a formal variable before taking derivatives and as a cohomology class afterwards.
	\end{remark}
	In order to be able to work with the $\theta_{k,l}$ as above, we will need the following Lemma
	\begin{lemma}\label{lemma: alphabetaintegrals}
		For any tuple $\bmm = (m_i)_{i=1}^n$ of nonnegative integers and tuples $\ba^1 =\left(a_{i}^1\right)_{i=1}^{s_1},\ldots,\ba^g = \left(a_{i}^g\right)_{i=1}^{s_g}, \bb^1 = \left(b_{i}^1\right)_{i=1}^{t_1},\ldots,\bb^g = \left(b_{i}^g\right)_{i=1}^{t_g},\bc^{1}=\left(c_{i}^1\right)_{i=1}^t,\bc^{2}=\left(c^{2}_{i}\right)_{i=1}^t$ of numbers in $\Set{1,\ldots,n}$ we have
		\begin{align*}
			\int_{C^{(\bmm)}}\prod_{i=1}^n u_i^{n_i} \prod_{j=1}^g \left(\prod_{i=1}^{r_j} \alpha_j^{a_{i}^j} \prod_{i=1}^{s_j}\beta_j^{b_{i}^j}\right)\prod_{i=1}^t\theta_{c_i^{1},c_i^{2}}\prod_{i_1,i_2=1}^n e^{z_{i_1,i_2}\theta_{i_1,i_2}} = 0\\
		\end{align*}
		unless $s_l=t_l$ for all $l$ in which case the integral equals
		\begin{align*}
			(-1)^{\sum_{i=1}^g r_i\cdot (r_i-1)/2} \sum_{\coprod_{i=1}^g S_i = \Set{1,\ldots,t}}[\mathbf{u}^{\bmm}]\prod_{i=1}^n u_i^{n_i+g} \detty{M}^g \prod_{i=1}^g  \detty{(M^{-1})_{\ba^i\sqcup \bc_{S_i}^{1},\bb^i\sqcup \bc_{S_i}^{2}}}			
		\end{align*}
		where $\sqcup$ denotes the concatenation of tuples and for any set $S = \Set{s_1<\ldots<s_l}$ we defined $\bc_S^j \coloneqq (c^j_{s_i})_{i=1}^l$. Furthermore we set $M = \left(\delta_{i,j}/u_i+z_{j,i}\right)_{1\leq i,j\leq n}$ and for any matrix $N = (n_{i,j})_{i,j}$ and tuples of indices $\mathbf{a} = (a_l)_l$, $\mathbf{b} = (b_l)_l$ we denote $N_{\mathbf{a},\mathbf{b}}\coloneqq \left(n_{a_i,b_j} \right)_{i,j}$.
	\end{lemma}
	\begin{proof}
		Let us first examine the case $t=0$. In Remark \ref{rem: replace_by_x} we noted how to compute integrals of the above kind.
		First, we have to express the product of classes pulled back from $\prod_{i=1}^n \Pic^0(C)$ in terms of monomials in the $\alpha_l^i,\beta_l^i$. Then we delete all unbalanced monomials and replace all balanced ones by the appropriate powers of $u_i$. Furthermore, we can write
		\begin{align}\label{eqn: helpprod}
			\prod_{l=1}^g \left(\prod_{i=1}^{r_l} \alpha_l^{a_{l,i}} \prod_{i=1}^{s_l}\beta_l^{b_{l,i}}\right)\prod_{i_1,i_2=1}^n e^{z_{i_1,i_2}\theta_{i_1,i_2}}= \prod_{l=1}^g \left(\prod_{i=1}^{r_l} \alpha_l^{a_{l,i}} \prod_{i=1}^{s_l}\beta_l^{b_{l,i}}\prod_{i_1,i_2=1}^n e^{z_{i_1,i_2}\alpha_l^{i_1}\beta_l^{i_2}} \right)
		\end{align}
		and the different factors of the outer product do not influence each other during this process. We can therefore assume $g=1$. Any $z$-monomial in the above product now corresponds to a directed graph on the vertices $1,\ldots,n$ as the occurrence of a $z_{i_1,i_2}\alpha^{i_1}\beta^{i_2}$ can be viewed as an edge from $i_2$ to $i_1$. The balancing condition is equivalent to the graph consisting of two parts:
		\begin{enumerate}
			\item a set of vertex-disjoint non-repeating directed cycles $\mathbf{C}$ on the vertices not in $\ba\cup \bb$
			\item a set of vertex-disjoint non-self-intersecting directed paths $\mathbf{P}$ which start at the entries in $\ba$ and end at the entries of $\bb$. Furthermore, we have $r=s$ and therefore there must be a permutation $\sigma\in S_{r}$ so that the path starting at $a_i$ ends at $b_{\sigma(i)}$ or $a_i=b_{\sigma(i)}$.
		\end{enumerate}
		Moreover, both sets are not allowed to have common vertices.
		The pair $\mathbf{F} = (\mathbf{P},\mathbf{C})$ is precisely what is referred to in \cite[Definition 2.3]{Talaska} as a \textit{self-avoiding flow} on the complete directed graph on $\Set{1,\ldots,n}$ connecting $\ba$ to $\bb$. By weighting each edge $e = (i\to j)$ by $\wt(e) = -u_i z_{j,i}$ and taking into account the signs produced by rearranging the $\alpha$'s and $\beta$'s, we can therefore replace \eqref{eqn: helpprod} by 
		\[
		(-1)^{r (r-1)/2}\cdot\sum_{\substack{(\mathbf{F},\sigma)\\ \text{as above}}} \sgn(\mathbf{F})\cdot \prod_{e\in E(\mathbf{F})} \wt(e)
		\]
		where $\sgn(\mathbf{F})=\sgn(\mathbf{P})\cdot \sgn(\mathbf{C})$ with $\sgn(\mathbf{P})=\sgn(\sigma)$ and $\sgn(\mathbf{C})=(-1)^{\lvert \mathbf{C}\rvert}$ and $(-1)^{r (r-1)/2}$ the sign that arises out of permuting the factors of $\prod_{i=1}^r \left(\alpha^{a_i}\beta^{b_i}\right)$ back into the order in which they appear in \eqref{eqn: helpprod}. By \cite[Theorem 2.5]{Talaska} this is equal to 
		\[ 
		(-1)^{r (r-1)/2}\cdot \left(\sum_{\mathbf{C}} \sgn(\mathbf{C})\prod_{e\in E(\mathbf{C})}\wt(e)\right)\cdot \detty{\left(\widetilde{M}^{-1}\right)_{\ba,\bb}}
		\]
		where $\widetilde{M} = \left(\delta_{i,j}+u_i z_{j,i}\right)_{1\leq i,j\leq n}$. Moreover, the second factor is easily seen to be $\detty{\widetilde{M}} = \prod_{i=1}^n u_i \cdot \detty{M}$ and so our expression is
		\[
		(-1)^{r (r-1)/2}\cdot\prod_{i=1}^n u_i \cdot \begin{Vmatrix} M \end{Vmatrix}\cdot \detty{\left(M^{-1}\right)_{\ba,\bb}}
		\]
		which establishes the claim in the case $t=0$. For the general claim one simply takes derivatives in $z_{i_1,i_2}$. For this one uses Jacobi's identity \cite[Lemma A.1(e)]{datweirduselesspaper} which says that for any tuples $\ba = (a_1<\ldots <a_s),\bb=(b_1<\ldots<b_s)$ of elements in $\Set{1,\ldots,n}$:
		\[
			\detty{M} \cdot \detty{\left(M^{-1}\right)_{\ba,\bb}} = (-1)^{\sum_{i=1}^r (a_i+b_i)}\detty{M_{\bb^c,\ba^c}}
		\]
		and therefore
		\[
		\frac{\partial}{\partial z_{i_1,i_2}}\left(\detty{M} \cdot \detty{\left(M^{-1}\right)_{\ba,\bb}}\right) = \begin{cases}
			\detty{M}  \cdot \detty{\left(M^{-1}\right)_{\ba\sqcup (i_1),\bb\sqcup (i_2)}},&\text{ if } i_1\not\in \ba \text{ and } i_2\not\in \bb \\
			0, &\text{ else}
		\end{cases}
		\]
		which holds even when the entries of $\ba$ and $\bb$ are not in increasing order.
	\end{proof}
	\subsection{Proof of Theorems \ref{thm: main}, \ref{thm: structurethm}\eqref{conj: rationality},\eqref{conj: powerminusone} and  \ref{thm: eigenvalues}}\label{sect: Proof of main}
	We now use the results of Section \ref{sect: IntTheorySymm} to derive Theorem \ref{thm: main} from Proposition \ref{prop: localization}. In addition to Theorem \ref{thm: main} we will show the following reformulation which will be more useful when proving Theorem \ref{thm: structurethm}\eqref{conj: poles} later on.
	\begin{theorem}\label{thm: mainthmproof}
		We have 
		\begin{align*}
			\langle \ch_{z_1}(\gamma_1) \ldots \ch_{z_n}(\gamma_n) \rangle_{d}^X = p^{d(1-g)}\sum_{\lambda\vdash d} \eval{\langle \ch_{z_1}(\gamma_1) \ldots \ch_{z_n}(\gamma_n) \rangle_{\lambda}^{X}}_{p_\Box = p}
		\end{align*}
			where $\langle\dots\rangle^X_\lambda\in\QQ(t_1,t_2)[[\bp]]$ is the power series in $\bp = (p_\Box)_{\Box\in\lambda}$ determined by super-commutativity and
		\begin{align}\label{eqn: P_lambda def}
			\begin{split}
			&\left\langle \prod_{i=1}^a\ch_{x_i}(1) \cdot\prod_{i=1}^b\ch_{y_i}(\pt)\cdot \prod_{l=1}^g \Big(\ch_{z_1^l}(\alpha_l)\ch_{w_1^l}(\beta_l)\ldots\ch_{z_{c_l}^l}(\alpha_l)\ch_{w_{c_l}^l}(\beta_l)\Big) \right\rangle_\lambda^X \\
			&\coloneqq \prod_{i=1}^a x_i\cdot\sum_{\substack{\coprod_{i=-1}^g S_i = \Set{1,\ldots,a}\\ \Box_i\in\lambda\text{ for }i\in S_{-1}}}\prod_{i\in S_-1}\left(p_{\Box_i}\frac{\partial}{\partial p_{\Box_i}}\right) \prod_{i\in S_{-1}} E(x_i,\B^\lambda_{\Box_i}) \\ &\cdot \prod_{i\in S_0} \Big(l_1\mathfrak{B}(x_i t_1)+l_2\mathfrak{B}(x_i t_2)\Big)E(x_i,\Bf^\lambda)
			\cdot\prod_{i=1}^b E(y_i,\Bf^\lambda)\\
			&\cdot \prod_{i=1}^g \left(\bz^i,\bw^i;\bx_{S_i}\mid \Bf^\lambda\right)_{\widetilde{M}(\bp)}\cdot A(\Bf^\lambda)^{g-1}\cdot B_1(\Bf^\lambda)^{l_1} \cdot B_2(\Bf^\lambda)^{l_2}
			\end{split}
		\end{align}
		where \[
		\widetilde{M}(\bp) = \left(p_\Box\frac{\partial \B^\lambda_{\Box'}(\bp) }{\partial p_\Box}\right)_{\Box,\Box'\in\lambda}\] and $\Bf^\lambda = \left(\B^\lambda_{\Box}(\bp)\right)_{\Box\in\lambda}$ is as in Theorem \ref{thm: multivariablebetheunique}.
	\end{theorem}
	\begin{proof}[Proof of Theorem \ref{thm: main} and Theorem \ref{thm: mainthmproof}]
		First, we express \eqref{eqn: appliedlocalizationformula} in terms of the cohomology classes introduced in Section \ref{sect: IntTheorySymm}. For this we fix a summand corresponding to a partition $\lambda$ and a tuple $\bmm = (m_{\lambda/\mu})_{\lambda/\mu}$ as in Proposition \ref{prop: localization} and let $\bn \coloneqq \lvert\bmm\rvert$. It follows from \eqref{eqn: Ftilde} that
		\begin{align*}
			\ch_z(\widetilde{\BF}) = \sum_{(i,j)\in \lambda} \exp\left(z\left(\CD_{i,j}-(i l_1+j l_2)\pt_C -it_1-jt_2\right)\right)
		\end{align*}
		with $\CD_{i,j} = \sum_{(i,j)\in \lambda/\mu} \CD_{\lambda/\mu}$ and $\CD_{\lambda/\mu}\subset C\times C^{(\bmm)}$ the universal divisors.
		We recall from \eqref{eqn: Duniv} that 
		\[
		\CD_{\lambda/\mu} = m_{\lambda/\mu}\cdot \pt_C + u_{\lambda/\mu} + \gamma_{\lambda/\mu}
		\]
		where 
		\[
		\gamma_{\lambda/\mu} = \sum_{l=1}^g \left(\beta_l^1\times \alpha_l^{\lambda/\mu}-\alpha_l^1\times \beta_l^{\lambda/\mu}\right).
		\]
		It is easy to check that 
		\[
		\gamma_{\lambda/\mu_1}\gamma_{\lambda/\mu_2} = -\pt_C\cdot \left(\theta_{\lambda/\mu_1,\lambda/\mu_2}+\theta_{\lambda/\mu_2,\lambda/\mu_1}\right)
		\]
		which gives
		\begin{align*}
			\ch_z(\widetilde{\BF}) &= \sum_{(i,j)\in\lambda}\exp\left(z\sum_{(i,j)\in \lambda/\mu} (m_{\lambda/\mu}\pt_C+u_{\lambda/\mu}+\gamma_{\lambda/\mu})-(il_1+jl_2)\pt_C-it_1-jt_2\right)\\
			&=\sum_{(i,j)\in\lambda}\exp\left(z\sum_{(i,j)\in \lambda/\mu} (u_{\lambda/\mu}+\gamma_{\lambda/\mu})-it_1-jt_2\right)\left(1+z(n_{(i,j)}-il_1-jl_2)\pt_C\right)\\
			&=\sum_{(i,j)\in\lambda}\exp\left(z\sum_{(i,j)\in \lambda/\mu} u_{\lambda/\mu}-it_1-jt_2\right)\left(1+z\sum_{(i,j)\in \lambda/\mu}\gamma_{\lambda/\mu} - z^2\pt_C\sum_{(i,j)\in \lambda/\mu_1,\lambda/\mu_2}\theta_{\lambda/\mu_1,\lambda/\mu_2}\right)\\
			&\cdot\left(1+z(n_{(i,j)}-il_1-jl_2)\pt_C\right)\\
			&=\sum_{(i,j)\in\lambda}\exp\left(z\sum_{(i,j)\in \lambda/\mu} u_{\lambda/\mu}-it_1-jt_2\right)\left(1+z\sum_{(i,j)\in \lambda/\mu}\left(\gamma_{\lambda/\mu} + (n_{(i,j)}-il_1-jl_2)\pt_C\right)\right.\\
			& \left.- z^2\pt_C\sum_{(i,j)\in \lambda/\mu_1,\lambda/\mu_2}\theta_{\lambda/\mu_1,\lambda/\mu_2}\right).
		\end{align*}
		Finally, using 
		\begin{align*}
			\gamma_{\lambda/\mu}\cdot \alpha_l^1  = \pt_C\cdot \alpha_l^{\lambda/\mu}\\
			\gamma_{\lambda/\mu}\cdot \beta_l^1 = \pt_C\cdot \beta_l^{\lambda/\mu}
		\end{align*}
		and \eqref{eqn: GRRondescendents} we see
		\begin{align*}
			\ch_z(1) &= z\sum_{(i,j)\in\lambda}E(z,\B_{i,j})\cdot \Bigg( n_{i,j}-il_1-jl_2 + l_1\CB(zt_1)+l_2\CB(zt_2)  \\
			&-z\sum_{(i,j)\in \lambda/\mu_1,\lambda/\mu_2}\theta_{\lambda/\mu_1,\lambda/\mu_2}\Bigg)\\
			\ch_z(\alpha_l) &= z\sum_{(i,j)\in\lambda}E(z,\B_{i,j})\cdot \sum_{(i,j)\in \lambda/\mu} \alpha_l^{\lambda/\mu}\\
			\ch_z(\beta_l) &=z\sum_{(i,j)\in\lambda}E(z,\B_{i,j})\cdot \sum_{(i,j)\in \lambda/\mu} \beta_l^{\lambda/\mu}\\
			\ch_z(\pt) &= E(y,\Bf)
		\end{align*}
		where we wrote $\Bf = (Y_\Box)_{\Box\in\lambda}$ for
		\[
		\B_{i,j} \coloneqq -i t_1 - j t_2 + \sum_{(i,j)\in \lambda/\mu} u_{\lambda/\mu}
		\]
		which using \eqref{eqn: Yoverlinedefi} can be written as $\Bf = \overline{\Bf}^\lambda(\bu)$ with $\bu = (u_{\lambda/ \mu})_{\lambda/\mu}$.
		Using Lemma \ref{usefulcohclass} we can further expand the denominator:
		\[
		\frac{1}{e(N_{\bmm})} = \overline{A}^{g-1}  \cdot \overline{B}_1^{l_1} \cdot \overline{B}_2^{l_2} \cdot \prod_{\lambda/\mu} \overline{F}_{\lambda/\mu}^\lambda(\Bf)^{-m_{\lambda/\mu}}\cdot h
		\]
		for 
		\begin{align*}
		\overline{A}(\Bf) &\coloneqq \prod_{0\neq (i,j)\in \lambda} \B_{i,j} \cdot \prod_{(i,j)\in\lambda} (t_1+t_2-\B_{i,j}) \cdot \prod_{\substack{(i,j),(k,l)\in\lambda\\0\leq a,b\leq 1\\(k,l)\neq (i+a,j+b)}} (at_1+bt_2 + \B_{k,l}-\B_{i,j})^{(-1)^{a+b+1}},\\
		\overline{B}_1(\Bf)&\coloneqq \prod_{(i,j)\in \lambda} \frac{\B_{i,j}^i}{(t_1+t_2-\B_{i,j})^{i+1}} \prod_{\substack{(i,j),(k,l)\in\lambda\\0\leq a,b\leq 1}} (at_1+bt_2 + \B_{k,l}-\B_{i,j})^{(i-k+a)(-1)^{a+b}}\\
		\overline{B}_2 (\Bf)&\coloneqq \prod_{ (i,j)\in \lambda} \frac{\B_{i,j}^j}{(t_1+t_2-\B_{i,j})^{j+1}} \prod_{\substack{(i,j),(k,l)\in\lambda\\0\leq a,b\leq 1}} (at_1+bt_2 + \B_{k,l}-\B_{i,j})^{(j-l+b)(-1)^{a+b}}\\
		h &\coloneqq \exp\left( \sum_{\lambda/\mu,\lambda/\mu'} \theta_{\lambda/\mu,\lambda/\mu'}\cdot \frac{\partial \overline{F}^\lambda_{\lambda/\mu}(\Bf)}{\partial u_{\lambda/\mu'}}/\overline{F}^\lambda_{\lambda/\mu}(\Bf)\right)
		\end{align*}
		where $\overline{F}^\lambda_{\lambda/\mu}(\Bf)$ is as in \eqref{eqn: Foverlinedef}.
		Using Lemma \ref{lemma: alphabetaintegrals} we can see that the contribution of $\bmm$ is the coefficient of $\bu^\bmm$ in 
		\[
			\overline{P}_{\lambda}(\bu)\cdot \prod_{\lambda/\mu} \overline{F}_{\lambda/\mu}^\lambda(\Bf)^{-m_{\lambda/\mu}}\cdot \prod_{\lambda/\mu} u_{\lambda/\mu}\cdot \detty{M_{\text{skew}}(\bu)}
		\] 
		where
		\begin{align}\label{eqn: Pbar_lambdadef}
			\begin{split}
			\overline{P}_\lambda(\bu) &=\prod_{i=1}^a x_i \cdot \sum_{\substack{\coprod_{i=0}^g S_i = \Set{1,\ldots,a}\\ (i_n,j_n)\in\lambda\text{ for }n\in S_{0}}}\prod_{n\in S_0} E(x_n,\B_{(i_n,j_n)}) \left(n_{(i_n,j_n)}-i_n t_1 - j_n t_2 + l_1\CB(x_n t_1)+l_2\CB(x_n t_2)\right)\\
			&\cdot\prod_{i=1}^b E(y_i,\Bf)
			\cdot \prod_{i=1}^g \left(\bz^i,\bw^i;\bx_{S_i}\right)'\\ 
			&\cdot \left(\overline{A}(\Bf)\cdot\prod_{\lambda/\mu} u_{\lambda/\mu}\cdot\detty{M_{\text{skew}}(\bu)}\right)^{g-1}\cdot \overline{B}_1(\Bf)^{l_1} \cdot \overline{B}_2(\Bf)^{l_2}.
			\end{split}
		\end{align}
		where 
		\begin{align*}
			\left(\bz,\bw;\bx\right)' \coloneqq& (-1)^n \sum_{\substack{\boldsymbol\mu^{(j)}=(\lambda/\mu_i^{(j)})_i\\ j=1,2,3,4\\\Box^{(1)}_i\in\lambda/\mu_i^{(1)},\Box^{(2)}_i\in\lambda/\mu_i^{(2)}\\ \Box^{(3)}_i\in\lambda/\mu_i^{(3)},\lambda/\mu_i^{(4)}}} \detty{\left(M_{\text{skew}}(\bu)^{-1}\right)_{\boldsymbol\mu^{(1)}\sqcup\boldsymbol\mu^{(3)};\boldsymbol\mu^{(2)}\sqcup\boldsymbol\mu^{(4)}}}\\ &\cdot\prod_{i=1}^m z_i w_i E(z_i,\B_{\Box^{(1)}_i}) E(w_i,\B_{\Box^{(2)}_i})\cdot \prod_{i=1}^n x_i E(x_i,\B_{\Box^{(3)}_i})
		\end{align*}
		and
		\[
		M_{\text{skew}}(\bu) \coloneqq \left(\frac{\delta_{\lambda/\mu,\lambda/\mu'}}{u_{\lambda/\mu}}+\frac{\partial \overline{F}^\lambda_{\lambda/\mu'}(\Bf) / \partial u_{\lambda/\mu}}{\overline{F}^\lambda_{\lambda/\mu'}(\Bf)} \right)_{\lambda/\mu,\lambda'/\mu'}.
		\]
		By using multivariate Lagrange inversion \cite[Theorem A]{Goodygoodgood} we see that this is the same as the coefficient of $\overline{\bp}^\bmm$ in $\overline{P}_\lambda(\bu(\overline{\bp}))$ where $\bu(\overline{\bp}) = \left(u_{\lambda/\mu}(\overline{\bp})\right)_{\lambda/\mu}$ is the unique power series in $\overline{\bp} = \left(\overline{p}_{\lambda/\mu}\right)_{\lambda/\mu}$ satisfying $u_{\lambda/\mu}(\overline{\bp}) = \CO (\overline{\bp}^{>0})$ and 
		\[
			\overline{p}_{\lambda/\mu} =u_{\lambda/\mu}(\overline{\bp})\cdot\overline{F}^\lambda_{\lambda/\mu}(\overline{\Bf}^{\lambda}(\bu(\overline{\bp}))).
		\]
		for any $\lambda/\mu$. Note however that neither the expression $\overline{P}_{\lambda}(\bu)$ nor the accompanying prefactor 
		\begin{equation}\label{eqn: p-prefactor}
			p^{\lVert\bmm\rVert-n(\lambda)\cdot l_1 - n(\bar{\lambda})\cdot l_2} = p^{\lvert\bn\rvert-n(\lambda)\cdot l_1 - n(\bar{\lambda})\cdot l_2}
		\end{equation}
		in \eqref{eqn: appliedlocalizationformula} depend on $\bmm$, but rather on $\bn$. Therefore we may consider the sum over all coefficients of $\bp^\bmm$ with $\bmm\vdash \bn$ for a fixed $\bn$. This sum can be expressed as:
		\[
			\sum_{\bmm\vdash\bn} [\overline{\bp}^\bmm] \overline{P}_\lambda(\bu(\overline{\bp})) = [\bp^\bn] \overline{P}_\lambda(\bu^\lambda(\bp))
		\]
		with $\bu^\lambda(\bp)$ as in Theorem \ref{thm: multivariablebetheunique}\eqref{cond: recursion_multi_S} and $\bp = (p_\Box)_{\Box\in\lambda}$ a new set of variables. As a result, $\Bf$ is the multivariate Bethe root $\Bf^\lambda(\bp)$ constructed in Theorem \ref{thm: multivariablebetheunique}.
		Using Lemmas \ref{lemma: bethedet} and \ref{prop: bethematrices} we see that 
		\[
			\overline{A}(\Bf)\cdot\prod_{\lambda/\mu} u^\lambda_{\lambda/\mu}(\bp)\cdot\detty{M_{\text{skew}}(\bu^\lambda(\bp))} = A(\Bf)\cdot \detty{M(\Bf)}
		\]
		and
		\[
			\left(\bz,\bw;\bx\right)' = \left( \bz,\bw;\bx\mid \Bf\right)_{M(\Bf)^{-1}}.
		\]
		By making these substitutions in $\overline{P}_{\lambda}$ we obtain an expression that only depends on $\Bf^\lambda$ as opposed to $\bu^\lambda$. In that expression we can furthermore replace $n_{(i,j)}$ by $n_{(i,j)}+it_1+jt_2$. Because of the shift in the prefactor \eqref{eqn: p-prefactor} and 
		\begin{align*}
			\overline{B}_1(\Bf)&= B_1(\Bf) \cdot \prod_{(i,j)\in\lambda} \overline{F}^\lambda_{\lambda/\mu}(\Bf)^i = \prod_{(i,j)\in\lambda} p_{(i,j)}^{i}\cdot B_1(\Bf)\\
			\overline{B}_2(\Bf)&= B_2(\Bf)\cdot \prod_{(i,j)\in\lambda} \overline{F}^\lambda_{\lambda/\mu}(\Bf)^j = \prod_{(i,j)\in\lambda} p_{(i,j)}^{j}\cdot B_2(\Bf)
		\end{align*}	
		we may write 
		\begin{align*}
			&\left\langle \prod_{i=1}^a\ch_{x_i}(1) \cdot\prod_{i=1}^b\ch_{y_i}(\pt)\cdot \prod_{l=1}^g \Big(\ch_{z_1^l}(\alpha_l)\ch_{w_1^l}(\beta_l)\ldots\ch_{z_{c_l}^l}(\alpha_l)\ch_{w_{c_l}^l}(\beta_l)\Big) \right\rangle_{d}^X\\
			& = p^{d(1-g)}\sum_{\lambda \vdash d} \eval{{P}_\lambda}_{p_\Box = p}
		\end{align*}
		for 
		\begin{align}\label{eqn: P'_lambda def}
			\begin{split}
			P_\lambda =& \prod_{i=1}^a x_i\cdot \sum_{\substack{\coprod_{i=-1}^g S_i = \Set{1,\ldots,a}\\ \Box_i\in\lambda\text{ for }i\in S_{-1}}}\prod_{i\in S_-1}\left(p_{\Box_i}\frac{\partial}{\partial p_{\Box_i}}\right) \prod_{i\in S_{-1}} E(x_i,\B^\lambda_{\Box_i}) \\ &\cdot \prod_{i\in S_0} \Big(l_1\mathfrak{B}(x_i t_1)+l_2\mathfrak{B}(x_i t_2)\Big)E(x_i,\Bf^\lambda)
			\cdot\prod_{i=1}^b E(y_i,\Bf^\lambda)
			\\&\cdot \left( \bz^i,\bw^i;\bx_{S_i}\mid \Bf^\lambda\right)_{M(\Bf^\lambda)^{-1}}
			\cdot A(\Bf^\lambda)^{g-1}\cdot B_1(\Bf^\lambda)^{l_1} \cdot B_2(\Bf^\lambda)^{l_2}.
			\end{split}
		\end{align}
		In order for this expression to agree with \eqref{eqn: Plambdamaindef} and \eqref{eqn: P_lambda def} we need
		\[
			M(\Bf^\lambda)^{-1} = \left(p_\Box\frac{\partial \B^\lambda_{\Box'}}{\partial p_{\Box}}\right)_{\Box,\Box'\in\lambda}
		\]
		which can be seen by applying partials $p_\Box\frac{\partial}{\partial p_\Box}$ to the Bethe equations \eqref{eqn: multivarBetheequn} i.e.
		\[
			p_{\Box'} = F_{\Box'}(\Bf^\lambda(\bp)).
		\]
	\end{proof}
	Using Theorem \ref{thm: main} we can now prove Theorem \ref{thm: structurethm} part \eqref{conj: rationality} and \eqref{conj: powerminusone}:
	\begin{proof}[Proof of Theorem \ref{thm: structurethm}\eqref{conj: rationality},\eqref{conj: powerminusone}]
		It follows from Theorem \ref{conj: allbethes} that the subscheme $\mathbf{Be}\subset \BA_{\QQ(t_1,t_2,p)}^d$ of fully admissible solutions of the Bethe equations has only finitely many points $K$-valued points. Since $K$ is algebraically closed it follows that $\mathbf{Be}$ is $0$-dimensional and thus all its points are $\overline{\QQ(t_1,t_2,p)}$-valued. As the absolute Galois-group \[G\coloneqq\Aut(\overline{\QQ(t_1,t_2,p)}/\QQ(t_1,t_2,p))\] preserves $\mathbf{Be}$ it thus also preserves the Bethe roots and acts on them by a combination of permuting the coordinates and permuting the partitions. But as \eqref{eqn: main fucker} is invariant under such permutations, all descendent invariants on local curves must be in $\overline{\QQ(t_1,t_2,p)}^G= \QQ(t_1,t_2,p)$ as desired. To prove part \eqref{conj: powerminusone} one first observes that $\Bf \mapsto (t_1+t_2-\B_i)_i$ is an isomorphism from $\mathbf{Be}$ to the base-change of $\mathbf{Be}$ by $\QQ(t_1,t_2,p)\to \QQ(t_1,t_2,p),p\mapsto p^{-1}$. Hence replacing $p$ by $p^{-1}$ amounts to replacing $\B_\Box^\lambda$ by $t_1+t_2-\B_\Box^\lambda$ since permutations of the partitions or boxes do not matter. Therefore we need to look at all the factors in \eqref{eqn: Plambdamaindef} and see how $\Bf \mapsto (t_1+t_2-\B_i)_i$ changes them. It is easy to see that $A(\Bf)$ and $M(\Bf)$ stay invariant under this substitution while $B_i(\Bf)$ gets replaced by $\prod_j F_j(\Bf)^{-1} B_i(\Bf)$. In case no descendents of $1$ are present this will give part of the $p^{-d_\beta}$ prefactor noting that $d_{d\cdot [C]}(X) = d\cdot (l_1+l_2+2-2g)$. In general we will get extra summands from 
		\[
		\left[\nabla^{\Bf}_i,\prod_{j} F_j(\Bf)^{-1}\right]  = -\prod_{j} F_j(\Bf)^{-1}
		\]
		which can be absorbed into the product over $S_0$ by using
		\[
		-\CB(x)-1 = \CB(-x).
		\]
		Finally, all descendent variables get negated because of \[E(z,t_1+t_2-X) = E(-z,X).\] 
	\end{proof}
	One can also deduce Theorem \ref{thm: eigenvalues} from Theorem \ref{thm: main} as follows:
	\begin{proof}[Proof of Theorem \ref{thm: eigenvalues}]
		Using the degeneration formula \cite[Section 6]{LiTheDegen} we have
		\[
		M(z_1,\ldots,z_n) = M(z_1)\ldots M(z_n)
		\]
		and hence the claim only has to be shown for $n=1$.
		Furthermore, we can degenerate $\BC^2\times E$ to a circle formed by copies of $\BC^2\times \BP^1$ where $\BC^2\times \Set{\infty}$ in each copy is identified with $\BC^2\times \Set{0}$ inside the next copy. This gives
		\begin{equation}\label{eqn: traceshit}
			\Tr \big[ M(z_1,\ldots,z_n)\big] = \langle \ch_{z_1}(\pt)\ldots\ch_{z_n}(\pt) \rangle_d^{\BC^2\times E} = \sum_{\lambda\vdash d} \prod_{i=1}^n\sum_{\Box\in\lambda} E(z_i,\B_\Box^\lambda)
		\end{equation}
		where we used Theorem \ref{thm: main} in the last equality.
		In particular, if we define $M_k$ so that
		\[
		\frac{t_1 t_2}{(1-e^{-t_1 z})(1-e^{-t_2 z})}M(z) = \sum_{k\geq 0} M_k \frac{z^k}{k!}
		\]
		then we get
		\[
		\Tr \big[ M_k^n\big] = \sum_{\lambda\vdash d} \left(\sum_{\Box\in\lambda} \left(\B_\Box^\lambda\right)^k\right)^n.
		\]
		It follows that the eigenvalues of $M_k$ must be the sums
		\[
		\sum_{\Box\in\lambda} \left(\B_\Box^\lambda\right)^k
		\]
		for $\lambda\vdash d$. We now claim that one can choose $k$ so that all of these power sums are distinct and the power sum corresponding to $\lambda$ has an eigenvector of the shape $v_\lambda = [\lambda]+\CO(p)$. For this we recall that
		\[
			P_d(\BC^2\times\BP^1/\Set{0,\infty},d) = \Hilb^d(\BC^2)
		\]
		and $\eval{M_k}_{p=0}$ corresponds to multiplication by $k!\cdot\ch_k(\pi_*\CO_Z)$ where $Z\subset \BC^2\times\Hilb^d(\BC^2)$ is the universal 0-dimensional subscheme. Furthermore, the fixed points $\lambda\in\Hilb^d(\BC^2)$ form an eigenbasis for multiplication by any class. Since 
		\[
			H^0\left(\CO_\lambda\right) = \bigoplus_{(i,j)\in\lambda} \BC\cdot \bt_1^{-i} \bt_2^{-j} 
		\] 
		 it follows that $\eval{M_k}_{p=0}$ has eigenvalue $\sum_{(i,j)\in\lambda} (-it_1-jt_2)^k$ at $[\lambda]$. To prove the distinctness on power sums we first specialize $t_1,t_2$ so that the $-it_1-jt_2$ for $0\leq i,j\leq d-1$ are distinct nonnegative numbers. One can show that any two sets of nonnegative numbers whose power sums agree for infinitely many powers must be equal\footnote{This follows from $\max(S) = \lim\limits_{k\to\infty}\left(\sum_{i\in S} i^k\right)^{1/k}$ for any non-empty finite $S\subset\BR_{\geq 0}$.} - hence any $k_0\gg 0$ will work. As a result,  $M_{k_0}$ has simple spectrum since $\eval{M_{k_0}}_{p=0}$ has and is therefore diagonalizable. Furthermore, for any $\lambda$ we have 
		\[
			\eval{\sum_{\Box\in\lambda} \left(\B_\Box^\lambda\right)^{k_0}}_{p=0} = \sum_{(i,j)\in\lambda} (-it_1-jt_2)^{k_0}
		\] 
		by Theorem \ref{thm: Bethechar}\eqref{cond: innit} and so we must have $\eval{v_\lambda}_{p=0} = [\lambda]$ for an appropriate eigenvector $v_\lambda$ with eigenvalue $\sum_{\Box\in\lambda} \left(\B_\Box^\lambda\right)^{k_0}$. It now remains to show that the $v_\lambda$ form an eigenbasis of $M(z)$ with eigenvalues as claimed. For this we note that
		\[
		M(z_1)M(z_2) = M(z_1,z_2) = M(z_2,z_1) = M(z_2)M(z_1)
		\]
		by degeneration and hence $M(z)$ commutes with $M_{k_0}$. Therefore the $v_\lambda$ indeed diagonalize $M(z)$ and we can write $a_\lambda(z) $ for the eigenvalue at $v_\lambda$. It follows that
		\[
		\Tr \big[ M(z) M_{k_0}^n\big] = \sum_{\lambda\vdash d} a_\lambda(z) \left(\sum_{\Box\in\lambda} \left(\B_\Box^\lambda\right)^{k_0}\right)^n
		\]
		for any $n\geq 0$. But by \eqref{eqn: traceshit} this must also equal 
		\[
		\sum_{\lambda\vdash d} \left(\sum_{\Box\in\lambda} E(z,\B_\Box^\lambda)\right) \left(\sum_{\Box\in\lambda} \left(\B_\Box^\lambda\right)^{k_0}\right)^n
		\]
		and as the eigenvalues of $M_{k_0}$ are pairwise distinct, this implies
		\[
		a_\lambda(z) = \sum_{\Box\in\lambda} E(z,\B_\Box^\lambda)
		\]
		by the invertibility of the Vandermonde matrix.
	\end{proof}
	\begin{remark}
		One could have shortened the proof somewhat by using that $M_3$ has simple spectrum \cite[Proof of Corollary 1]{QcohHilb}. Recalling that stable pairs on $\BC^2\times C$ are the same thing as quasi-maps from $C$ to $\Hilb^d(\BC^2)$ (c.f. \cite[Exercise 4.3.22]{OkounkovLectureNotes}), this corresponds to the fact that the quantum cohomology of $\Hilb^d(\BC^2)$ is generated by divisors. However, we chose to circumvent this fact since the analogous claim for general Nakajima quiver varieties is still a conjecture \cite[Question 1]{QgroupQcoh}. This makes it possible to repeat the above proof for quasi-maps to quiver varieties, which gives a new proof of the fact that the spectrum of quantum multiplication is described by solutions of Bethe equations. See \cite{ShamelessPlug} for more details and consequences.
	\end{remark}
	
	\section{More on Bethe roots}\label{sect: Bethe}
	We now want to prove Theorem \ref{thm: Bethechar} and Theorem \ref{thm: structurethm}\eqref{conj: poles}. As mentioned in the introduction, one would like to prove Theorem \ref{thm: structurethm}\eqref{conj: poles} by first showing that $\B_\Box^\lambda(p)$ is holomorphic at $p=0$ and can be locally extended\footnote{By this we mean that it can be analytically continued to any simply connected open subset thereof. However, this continuation is usually not unique. Indeed, see \cite{procházka} for an numerical investigation of the monodromy.} to all of 
	\[\BC\setminus\Set{\zeta|(-\zeta)^n=1\text{ for some } 1\leq n\leq d }.\] Indeed, if all factors in \eqref{eqn: Plambdamaindef} were polynomial in $\Bf^\lambda$, then this would already be enough, however most are merely rational functions and so we have no a priori control over their poles. Luckily, one can circumvent this problem at the cost of working with a multivariate version of the Bethe roots and using \eqref{eqn: P_lambda def} instead of \eqref{eqn: Plambdamaindef}.
	\subsection{Multivariate Bethe roots}
	We begin by showing a multivariable generalization of Theorem \ref{thm: Bethechar}.
	For this let $k = \QQ(t_1,t_2)$ and $\lambda$ be a fixed partition. For the rest of this subsection we will fix a collection $\bp = (p_\Box)_{\Box\in\lambda}$ of possibly repeating variables which are otherwise free. Let $\nu$ be the non-archimedean valuation on $k[[\bp]]$ given for any 
	\[
		x =  \sum_{\substack{\bn = (n_\Box)_{\Box\in\lambda}\\n_\Box\geq 0}} a_{\bn}\cdot \bp^{\bn}\in k[[\bp]]
	\] 
	by
	\[
		\nu(x) \coloneqq \inf\Set{ m| \text{ there is an } \bn \text{ so that } a_{\bn}\neq 0 \text{ and } m=\sum_{\Box\in\lambda} n_\Box} \in \BN\cup\{\infty\}
	\] 
	Furthermore let $K$ be any field containing $k[[\bp]]$ equipped with an extension of $\nu$ which we will also denote by $\nu$. In particular if $n=1$ one may choose $K = \overline{\QQ(t_1,t_2)}\{\{p\}\}$ to be the field of Puiseux series with its canonical valuation.
	\begin{theorem}\label{thm: multivariablebetheunique}
		There is a tuple $\bY^{\lambda}(\bp) = (Y_{\Box}^\lambda(\bp))_{\Box\in\lambda}$ of power series $Y_\Box^\lambda(\bp)\in k[[\bp]]$ characterized uniquely via any of the following equivalent descriptions:
		\begin{enumerate}
			\item\label{item: multivar_half_alg_char} $\Bf^\lambda\in K^d$ is the unique tuple such that
			\begin{itemize}
				\item\label{cond: admissible_multi} it is \textit{admissible} in the sense that for any $\Box\in\lambda$ we have $\B_{\Box}^\lambda\not\in \{0,t_1+t_2\}$ and for $\Box\neq\Box'\in\lambda$ we have $\B_\Box^\lambda-\B_{\Box'}^\lambda\not\in\{t_1,t_2,t_1+t_2\}$.
				\item\label{item: initialterm_multi} one has
				\begin{equation}
				\nu\left(\B_{(i,j)}^\lambda  +it_1+jt_2\right)>0
				\end{equation}
				\item it satisfies the \textit{multivariate Bethe equations} i.e. for every $\Box\in\lambda$ we have
				\begin{align}\label{eqn: multivarBetheequn}
					p_\Box = F_{\Box}(\Bf^\lambda) 
				\end{align}
				where
				\[
				F_{\Box}(\Bf) = f_\Box(\Bf) \prod_{\Box\neq\Box'\in\lambda} g_{\Box',\Box}(\Bf)
				\]
				with 
				\[
					f_\Box(\Bf) = \frac{\B_{\Box}}{t_1+t_2-\B_{\Box}}
				\] 
				and
				\[
					g_{\Box,\Box'}(\Bf) = \prod_{\substack{0\leq a,b,c\leq 1\\(a,b)\neq (0,0)}} \left((-1)^c(at_1+bt_2)+\B_{\Box}-\B_{\Box'} \right)^{(-1)^{a+b+c}}
				\] 
			\end{itemize}
			\item\label{cond: recursion_multi_box} It can be written as
			\begin{align}
				\Bf^{\lambda}(\bp) = \widetilde{\Bf}^\lambda(\bv^\lambda) 			
			\end{align}
			where for any tuple $\bv = (v_\Box)_{\Box\in\lambda}$ we define $\widetilde{\Bf}^\lambda(\bv) = (\widetilde{\B}_\Box^\lambda(\bv) )_{\Box\in\lambda}$ by 
			\begin{equation}\label{eqn: Ytildef}
				\widetilde{\B}_{(i,j)}^\lambda(\bv)\coloneqq -it_1-jt_2+\sum_{\substack{\lambda/\mu\text{ conn. skew}\\ (i,j) \in \lambda/\mu}} \prod_{\Box\in \lambda/\mu} v_\Box
			\end{equation}
			and $\bv^{\lambda} = (v^\lambda_\Box)_{\Box\in\lambda}$ is the unique tuple of power series in $k[[\bp]]$ so that $v^\lambda_\Box = \CO(\bp^{>0})$ for all $\Box\in\lambda$ and
			\[
				p_{\Box} =v^\lambda_\Box\cdot\widetilde{F}^\lambda_{\Box}(\widetilde{\Bf}^{\lambda}(\bv^\lambda))
			\]
			where
			\[
				\widetilde{F}^\lambda_{\Box}(\Bf) = \tilde{f}_\Box(\Bf) \prod_{\Box\neq\Box'\in\lambda} \widetilde{g}_{\Box',\Box}(\Bf)
			\]
			with
			\[
				\tilde{f}_{\Box}(\Bf) \coloneqq \frac{\B_{\Box}^{1-\delta_{\Box,(0,0)}}}{t_1+t_2-\B_{\Box}}
			\]
			and 
			\[
				\widetilde{g}_{\Box,\Box'}(\Bf) = -\prod_{\substack{0\leq a,b,c\leq 1\\ (a,b)\neq (0,0)\\ \Box\neq \Box'+(-1)^c( a,b)}}\left(at_1+bt_2+(-1)^c (\B_{\Box}-\B_{\Box'} )\right)^{(-1)^{a+b+c}}.
			\]
			\item\label{cond: recursion_multi_S} It can be written as
			\begin{align}
				\Bf^{\lambda}(\bp) = \overline{\Bf}^\lambda(\bu^\lambda) 			
			\end{align}
			where for any tuple $\bu = (\bu_{\lambda/\mu})_{\lambda/\mu}$ we define $\overline{\Bf}^\lambda(\bu) = (\overline{\B}_{\Box}^\lambda(\bu))_{\Box\in\lambda}$ by
			\begin{equation}\label{eqn: Yoverlinedefi}
				\overline{\B}_{(i,j)}^\lambda(\bu)\coloneqq -it_1-jt_2+\sum_{\substack{\lambda/\mu\text{ conn. skew}\\ (i,j) \in \lambda/\mu}} u_{\lambda/\mu}
			\end{equation}
			and where $\bu^\lambda = (u^\lambda_{\lambda/\mu})_{\lambda/\mu}$ is the unique tuple of power series in $k[[\bp]]$ so that $u_{\lambda/\mu}^\lambda = \CO(\bp^{>0})$ for all $\Box\in\lambda$ and
			\[
			\prod_{\Box\in\lambda/\mu} p_{\Box} =u_{\lambda/\mu}^\lambda\cdot\overline{F}^\lambda_{\lambda/\mu}(\overline{\Bf}^{\lambda}(\bu^\lambda))
			\]
			where
			\begin{equation}\label{eqn: Foverlinedef}
			\overline{F}^\lambda_{\lambda/\mu}(\Bf) = \prod_{\Box\in\lambda/\mu}\tilde{f}_{\Box}(\Bf) \prod_{\substack{\Box\in\lambda/\mu\\\Box'\not\in\lambda/\mu}} \widetilde{g}_{\Box',\Box}(\Bf).
			\end{equation}
			\item\label{cond: closedform_multi} $\B_\Box^\lambda(\bp)$ is the coefficient of $\prod_{\Box\in\lambda} v_\Box^{-1}$ in
			\[
			\widetilde{Y}^\lambda_\Box(\bv)\cdot\begin{Vmatrix}
				\frac{\partial F_{\Box'}(\widetilde{\Bf}^{\lambda}(\bv))/\partial v_{\Box''}}{F_{\Box'}(\widetilde{\Bf}^{\lambda}(\bv))}
			\end{Vmatrix} \cdot\prod_{\Box'\in\lambda} \frac{1}{1-p_{\Box'}\cdot F_{\Box'}(\widetilde{\Bf}^{\lambda}(\bv))^{-1}} 
			\]
			all of whose $\bp$-coefficients turn out to be Laurent series in $\bv$. Furthermore, $\widetilde{\Bf}^{\lambda}(\bv)$ is as in \eqref{eqn: Ytildef}.
		\end{enumerate}
	\end{theorem}
	\begin{remark} 
		Theorem \ref{thm: Bethechar} follows from Theorem \ref{thm: multivariablebetheunique} by taking $\bp= (p,\cdots,p)$. In particular, the recursion of Theorem\ref{thm: Bethechar}\eqref{cond: recursion} is just Banach fixed point iteration. 
	\end{remark}
	\begin{proof}
		Indeed, the tuples in \eqref{cond: recursion_multi_box} and \eqref{cond: recursion_multi_S} are not just unique over $k[[\bp]]$, but also over $K$. We will only show this for \eqref{cond: recursion_multi_box} as \eqref{cond: recursion_multi_S} is similar. For this note that on the set
		\[
		\CU\coloneqq \Set{\bv = (v_\Box)_{\Box\in\lambda}\in K^d| \nu(v_\Box)>0 \text{ for all }\Box\in\lambda}
		\]
		the map $\bG\colon \CU\to\CU, \bv\mapsto (G_\Box(\bv))_{\Box\in\lambda}$ given by 
		\[
			G_\Box(\bv)\coloneqq p_{\Box}\cdot \widetilde{F}^\lambda_{\Box}(\widetilde{\Bf}^{\lambda}(\bv))^{-1}
		\]
		is a contraction with respect to the metric induced by $\nu$. Therefore the Banach fixed point theorem implies that there can only be at most one fixed point and since the complete subset $\CU\cap k[[\bp]]^d$ is invariant with respect to $\bG$ it follows that this fixed point exists and its coordinates must be power series.\\
		To show the uniqueness in \eqref{item: multivar_half_alg_char} and the equivalence of \eqref{item: multivar_half_alg_char}, \eqref{cond: recursion_multi_box} and \eqref{cond: recursion_multi_S} it will suffice to establish bijections (indeed we will give isomorphisms of varieties) between
		\[
		\mathbf{Be}\coloneqq\Set{ \Bf = (\B_\Box)_{\Box\in\lambda}\in K^d | \begin{array}{l} \Bf\text{ is admissible and} \\\text{for all }\Box\in\lambda \colon
				p_\Box = F_{\Box}(\Bf) \end{array}}
		\]
		and
		\[
		\widetilde{\mathbf{Be}}\coloneqq\Set{\bv = (v_\Box)_{\Box\in\lambda}\in K^d | \begin{array}{l}\widetilde{\Bf}^\lambda(\bv)\text{ is admissible and}\\\text{for all } \Box\in\lambda \colon p_{\Box} =v_\Box\cdot\widetilde{F}^\lambda_{\Box}(\widetilde{\Bf}^{\lambda}(\bv))\end{array} }
		\]
		and
		\[
		\overline{\mathbf{Be}}\coloneqq\Set{\bu = (u_{\lambda/\mu})_{\lambda/\mu}\in \prod_{\lambda/\mu} K | \begin{array}{l}\overline{\Bf}^\lambda(\bu)\text{ is admissible and}\\\text{for all } \lambda/\mu \colon \prod_{\Box\in\lambda/\mu} p_{\Box} =u_\Box\cdot\overline{F}^\lambda_{\Box}(\overline{\Bf}^{\lambda}(\bu))\end{array} }
		\]
		so that the conditions $\nu(Y_{i,j}+it_1+jt_2)>0$, $\nu(v_\Box)>0$ and $\nu(u_{\lambda/\mu})>0$ become equivalent and $\Bf = \widetilde{\Bf}^\lambda(\bv) = \overline{\Bf}^\lambda(\bu)$.
		We claim that
		\begin{align*}
		\widetilde{\mathbf{Be}} &\longleftrightarrow \overline{\mathbf{Be}}\\
		\bv &\longmapsto \bu(\bv) = (u_{\lambda/\mu}(\bv))_{\lambda/\mu}\\
		\bv(\bu) = (v_\Box(\bu))_{\Box\in\lambda} &\longmapsfrom \bu
		\end{align*} 
		for 
		\[
		v_\Box(\bu) \coloneqq p_{\Box}\cdot \widetilde{F}^\lambda_{\Box}(\overline{\Bf}^{\lambda}(\bu))^{-1}
		\]
		and
		\[
			u_{\lambda/\mu}(\bv) \coloneqq \prod_{\Box\in\lambda/\mu} v_\Box
		\]
		is one such bijection. Indeed, we have 
		\[
			\overline{\Bf}^\lambda(\bu(\bv)) = \widetilde{\Bf}^\lambda(\bv)
		\]
		which shows the well-definedness of $\bu(\bv)$ and $\bv(\bu(\bv))=\bv$. For the rest can use $\tilde{g}_{\Box,\Box'}(\Bf) = \tilde{g}_{\Box',\Box}(\Bf)^{-1}$ to see that
		\[
		\overline{F}^\lambda_{\lambda/\mu}(\Bf) = \prod_{\Box\in\lambda/\mu} \widetilde{F}^\lambda_{\Box}(\Bf)
		\]
		and therefore 
		\[
			u_{\lambda/\mu}(\bv(\bu)) = \prod_{\Box\in\lambda/\mu} \left(p_{\Box}\cdot \widetilde{F}^\lambda_{\Box}(\overline{\Bf}^{\lambda}(\bu))^{-1}\right) = \prod_{\Box\in\lambda/\mu} p_{\Box}\cdot \overline{F}^\lambda_{\lambda/\mu}(\overline{\Bf}^{\lambda}(\bu))^{-1} = u_{\lambda/\mu}
		\]
		which gives the rest.
		Finally, it follows from Lemma \ref{lemma: box iso bethe} that 
		\begin{align*}
			\widetilde{\mathbf{Be}} &\longrightarrow \mathbf{Be}\\
			\bv &\longmapsto \widetilde{\Bf}^\lambda(\bv) 
		\end{align*}
		is a bijection with inverse $\Bf\mapsto \widetilde{\bv}^\lambda(\Bf)$ and for the unique $\bv$ with $\nu(v_\Box) > 0$ one has $\nu\left(\widetilde{\B}^\lambda_{(i,j)}(\bv)+it_1+jt_2\right)>0$. For the converse observe that for any $\Bf = (\B_\Box)_{\Box\in\lambda}$ with $\nu\left(\B_{(i,j)}+it_1+jt_2\right)>0$ one has $\nu(v_{\Box}(\Bf)) = 1-\nu(\widetilde{F}^\lambda_{\Box}(\Bf)) = 1>0$ and therefore $\Bf$ is also unique.\\
		It remains to prove the formula in \eqref{cond: closedform_multi} for $\Bf^\lambda$ as in \eqref{cond: admissible_multi}, \eqref{cond: recursion_multi_box} and \eqref{cond: recursion_multi_S}. For this we note that the characterization \eqref{cond: recursion_multi_box} can be seen as an inversion of power series. As $\widetilde{F}^\lambda_\Box(\widetilde{\Bf}^\lambda(\mathbf{0}))\neq 0$ we can use multivariate Lagrange inversion in the shape of \cite[Theorem A]{Goodygoodgood} to conclude that for any tuple of natural numbers $\bn = (n_\Box)_{\Box\in\lambda}$ we have
		\begin{align*}
			[\bp^{\bn}] \B^\lambda_{\Box}(\bp)& = [\bv^{\bn}]\left( \widetilde{\Bf}^\lambda_\Box(\bv)\prod_{\Box'\in\lambda} \widetilde{F}^\lambda_{\Box'}(\widetilde{\Bf}^{\lambda}(\bv))^{-n_{\Box'}} \detty{\delta_{\Box',\Box''}+v_{\Box'}\frac{\partial \widetilde{F}^\lambda_{\Box''}(\widetilde{\Bf}^{\lambda}(\bv))/\partial v_{\Box'}}{\widetilde{F}^\lambda_{\Box''}(\widetilde{\Bf}^{\lambda}(\bv))}}\right)\\
			&= [\bv^{\bn-\mathbf{1}}]\left( \widetilde{\Bf}^\lambda_\Box(\bv)\prod_{\Box'\in\lambda} \widetilde{F}^\lambda_{\Box'}(\widetilde{\Bf}^{\lambda}(\bv))^{-n_{\Box'}} \detty{\frac{\partial F_{\Box''}(\widetilde{\Bf}^{\lambda}(\bv))/\partial v_{\Box'}}{F_{\Box''}(\widetilde{\Bf}^{\lambda}(\bv))}}\right)\\
			&= [\bv^{-\mathbf{1}}]\left( \widetilde{\Bf}^\lambda_\Box(\bv)\prod_{\Box'\in\lambda} F^\lambda_{\Box'}(\widetilde{\Bf}^{\lambda}(\bv))^{-n_{\Box'}} \detty{\frac{\partial F_{\Box''}(\widetilde{\Bf}^{\lambda}(\bv))/\partial v_{\Box'}}{F_{\Box''}(\widetilde{\Bf}^{\lambda}(\bv))}}\right)
		\end{align*}
		where we used $v_\Box\cdot \widetilde{F}^\lambda_{\Box}(\widetilde{\Bf}^{\lambda}(\bv)) = F^\lambda_{\Box}(\widetilde{\Bf}^{\lambda}(\bv))$ in the second and third equality which follows from Lemma \ref{lemma: box iso bethe}. After summing over $\bn$ we get the claimed formula.
		\end{proof}
	
	\subsection{Proof of Theorem \ref{thm: structurethm}\eqref{conj: poles}}\label{sect: Pf of pole shit}
		For the rest of this section we will fix an embedding $\QQ(t_1,t_2) \hookrightarrow \BC$. Let $\bp = (p_1,\ldots,p_d)$ to be a tuple of free variables (without repetitions). We then take $\Bf^\lambda(\bp)$ to be the multivariate Bethe roots described in Theorem \ref{thm: multivariablebetheunique} where $\bp$ is re-indexed by the boxes of $\lambda$ in an arbitrary way.
		Our first step is to show:
		\begin{lemma}\label{lemma: betheholopole}
			The coordinates $\Bf^\lambda_\Box(\bp)\in \BC [[\bp]]$ of the multivariate Bethe roots $\Bf^\lambda(\bp)$ are holomorphic near the origin and can be locally analytically continued to any point in the complement of 
			\[
				X = \Set{ \bp = (p_i)_i\in \BC^d | \text{there exists } \emptyset\neq S\subset \Set{1,\ldots,d} \text{ so that } \prod_{i\in S} (-p_i) = 1}
			\]
		\end{lemma}
		\begin{proof}
			We consider the subset $Z\subset \BC^d\times \BC^d$ consisting of $(\Bf,\bp)$ so that $\Bf$ is admissible, the multivariate Bethe equations
			\[
				p_i = F_i(\Bf)
			\]
			are satisfied and so that the Jacobian $\left(\frac{\partial F_i(\Bf)}{\partial \B_j}\right)_{i,j}$ is invertible. As a result, $Z$ is smooth of dimension $d$. Since the multivariate Bethe roots come as inversions of convergent power series they themselves converge in a small enough neighborhood. Furthermore, they are admissible and have a nonvanishing Jacobian determinant on the level of power series which implies that there is an open subset of $\BC^d$ on which they all give sections of $\pi_2\colon Z\to \BC^d$. Therefore $\pi_2$ is dominant and on the complement of some big enough proper algebraic subset $Y\subset\BC^d$ it is also finite \cite[Exercise II.3.7]{Hartshorne}, flat \cite[Tag 052A]{stacksproject}, unramified (characteristic 0) and hence a holomorphic covering. All Bethe roots therefore admit local analytic continuations to any point in this complement. \\
			By Riemann's extension theorem \cite[Chapter 1]{GriffithsHarris} one now only needs to show that these stay bounded when approaching any point in the complement of $X$. For this let $\bp^{(n)}$ be a sequence of points in $\BC^d\setminus Y$ converging to a point $\bp\in\BC^d\setminus X$ and assume that each $\B^\lambda_\Box(\bp^{(n)})$ either stays bounded or diverges to $\infty$. Let $S\subset \Set{1,\ldots,d}$ be the set of indices where the latter happens. We then have
			\[
				\prod_{i\in S} p^{(n)}_i = \prod_{i\in S} f_i(\Bf^\lambda(\bp^{(n)})) \prod_{\substack{i\in S\\ j\neq i}} g_{j,i}(\Bf^\lambda(\bp^{(n)}))= \prod_{i\in S} f_i(\Bf^\lambda(\bp^{(n)})) \prod_{\substack{i\in S\\ j\not\in S}} g_{j,i}(\Bf^\lambda(\bp^{(n)}))
			\]
			
			where we used $g_{j,i}(\Bf) = g_{i,j}(\Bf)^{-1}$ in the second equality. Since the right hand side converges to $(-1)^{\lvert S\rvert}$ we must have $S=\emptyset$.
		\end{proof}
		This lets us control the poles of most of the factors appearing in \eqref{eqn: P_lambda def}. The only ones that need to be dealt with in more detail are $A$, $B_1$ and $B_2$.
		\begin{lemma}
			For any $\lambda$ the germs $A(\Bf^\lambda(\bp))$, $B_1(\Bf^\lambda(\bp))$ and $B_2(\Bf^\lambda(\bp))$ can be locally analytically continued to all of $(\BC^*)^d\setminus X$ so that the latter two have no zeros. Here $X$ is as in Lemma \ref{lemma: betheholopole}.
			\end{lemma}
		\begin{proof}
			Viewing $A(\Bf)$ purely as an element of $\QQ(t_1,t_2,\B_1,\ldots,\B_d)$ it was observed in the proof of Theorem \ref{thm: mainthmproof} that for any partition $\lambda$ and any re-indexing of $\Bf$ by the boxes of $\lambda$ one can write
			\begin{align*}
				A(\Bf) = &\prod_{0\neq (i,j)\in \lambda} \B_{i,j} \cdot \prod_{(i,j)\in\lambda} (t_1+t_2-\B_{i,j}) \cdot \prod_{\substack{(i,j),(k,l)\in\lambda\\0\leq a,b\leq 1\\(k,l)\neq (i+a,j+b)}} (at_1+bt_2 + \B_{k,l}-\B_{i,j})^{(-1)^{a+b+1}}\\ &\cdot\prod_{\lambda/\mu} u_{\lambda/\mu}(\widetilde{\bv}^\lambda(\Bf))\cdot \begin{Vmatrix}
					M_{\text{skew}}(\bu(\widetilde{\bv}^\lambda (\Bf)))
				\end{Vmatrix}
			\end{align*}
			with notation as in Lemma \ref{prop: bethematrices}. It follows from this that the denominator of $A(\Bf)$ can therefore consist only of products of expressions of the shape $at_1+bt_2+(-1)^c(\B_{(i,j)}-\B_{(k,l)})$ where $(i,j)\neq (k+(-1)^c a,l+(-1)^c b)$. Note that these factors heavily depend on the indexing of the $\B_{(i,j)}$, however $A(\Bf)$ is symmetric in $\Bf$ and so the set of possible factors in the denominator also has to be invariant under index change. Indeed, this excludes all factors and as a result $A(\Bf)$ must actually be in $\QQ(t_1,t_2)[\B_1,\ldots,\B_d]$. Lemma \ref{lemma: betheholopole} now implies the first part of the claim.\\
			For the rest we will only examine $B_1$ as $B_2$ is similar. The function $B_1(\bp) \coloneqq B_1(\Bf^\lambda(\bp))$ is certainly meromorphic. We aim to show that it extends holomorphically to any point $\bp_0\in (\BC^*)^d\setminus X$ and is not zero there. For this we define an equivalence relation on $\Set{1,\ldots,d}$ by
			\[
				i\sim j\text{ if and only if } \B_i^\lambda(\bp_0) -\B_j^\lambda(\bp_0) \in \BZ\cdot t_1 + \BZ\cdot t_2.
			\] 
			Let $S_0,S_1,\ldots,S_n$ be the equivalence classes so that $S_0$ is the set of $i$ with $\B_i^\lambda(\bp_0)\in \BZ\cdot t_1 + \BZ\cdot t_2$ if there are such $i$ and otherwise we artificially set $S_0 = \emptyset$. We further choose elements $s_1\in S_1,s_2\in S_2,\ldots,s_n\in S_n$ and take $\iota\colon\Set{1,\ldots,d}\to \BZ$ to be the map so that for $i\in S_0$ one has 
			\[
				\B_i^\lambda(\bp_0)+\iota(i)t_1\in \BZ \cdot t_2
			\]
			and for any $i\in S_j$ with $j>0$ we need
			\[
				\B_i^\lambda(\bp_0)-\B_{s_j}^\lambda(\bp_0) +\iota(i) t_1 \in \BZ\cdot t_2.
			\] 
			For $\bp$ in a dense open subset of $\BC^d$ one now has:
			\begin{align*}
				\prod_{i=1}^d p_i^{\iota(i)} B_1(\bp) &= \prod_{i=1}^d F_i(\Bf^\lambda(\bp))^{\iota(i)} B_1(\bp)\\& = \pm\prod_{i=1}^d \frac{\B_i^\lambda(\bp)^{\iota(i)}}{(t_1+t_2-\B^\lambda_i(\bp))^{\iota(i)+1}}\\
				&\cdot \prod_{\substack{i\neq j\\0\leq a,b\leq 1}}\left(at_1+bt_2+\B^\lambda_j(\bp)-\B^\lambda_i(\bp)\right)^{(\iota(i)-\iota(j)+a)(-1)^{a+b}}.
			\end{align*}
			As a result, none of the factors in the second product are zero at $\bp_0$. The factors in the first product may only vanish for $i\in S_0$ so that $\B_i^\lambda(\bp_0)=0$ or $\B_i^\lambda(\bp_0)=t_1+t_2$. However, in the first case we have $\iota(i)=0$ and $\iota(i)=-1$ in the second - in each case the vanishing factor is removed. Therefore the whole product is holomorphic and non-vanishing near $\bp_0$.
		\end{proof}
		In case our local curve has genus at least $1$ it now follows from \eqref{eqn: P_lambda def} and the previous two Lemmas that $P_\lambda$ is holomorphic on $(\BC^*)^d\setminus X$ and the restriction along $p_\Box= p$ can only have poles at $p=0$ or where $-p$ is an $n$-th root of unity for $1\leq n \leq d$. This finishes the proof of Theorem \ref{thm: structurethm}\eqref{conj: poles} in this case. In order to similarly deduce the $g=0$ case we would need to know that $A(\Bf^\lambda(\bp))$ also never vanishes on $(\BC^*)^d\setminus X$, but it is not clear to us how to show this. However, since $\BP^1$ has only even cohomology classes we can use \cite[Theorem 5]{PaPixRat}, which establishes the pole statement in that case.
		\section{Auxiliary Lemmas}\label{sect: combi}
		Let $\QQ(t_1,t_2)\subset K$ be any field extension.
		\begin{lemma}\label{lemma: box iso bethe}
			The morphism
			\begin{align*}
				\BA_K^d\supset\Set{\Bf = (\B_\Box)_{\Box\in\lambda} | \Bf \text{ is admissible}}& \longrightarrow \BA_K^d\\
				\Bf &\longmapsto \widetilde{\bv}^\lambda(\Bf)
			\end{align*}
			where $\widetilde{\bv}^\lambda(\Bf)=(\widetilde{v}^\lambda_\Box(\Bf))_{\Box\in\lambda}$ with 
			\begin{equation}\label{eqn: v(Y) defn}
			\widetilde{v}^\lambda_\Box(\Bf) = \B_{\Box}^{\delta_{\Box,0}}\prod_{\substack{0\leq a,b,c\leq 1\\(a,b)\neq (0,0)\\ \Box'\coloneqq\Box+(-1)^c (a,b)\in\lambda}} \left(at_1+bt_2+(-1)^c (\B_{\Box'}-\B_{\Box})\right)^{(-1)^{a+b+c}}
			\end{equation}
			is an open immersion with image
			\[
				\CU = \Set{\bv = (v_\Box)_{\Box\in\lambda}| \widetilde{\Bf}^\lambda(\bv) \text{ is admissible}}
			\]
			on which $\widetilde{\Bf}^\lambda$ defined as in \eqref{eqn: Ytildef} is the inverse.
		\end{lemma}
		\begin{proof}
		It suffices to show  $\widetilde{\Bf}^\lambda(\widetilde{\bv}^\lambda(\Bf))=\Bf$. Indeed, this would imply that the Jacobian of $\Bf\mapsto \widetilde{\bv}^\lambda(\Bf)$ is invertible everywhere which makes the map étale and therefore flat. As it is also a monomorphism, \cite[Tag 06NC]{stacksproject} implies that it must be an open immersion and hence for any $\bv$ in its image $\widetilde{\Bf}^\lambda(\bv)$ must be admissible and $\widetilde{\bv}^\lambda(\widetilde{\Bf}^\lambda(\bv))=\bv$.\\
		Assume the claim has been shown for any partition of smaller size and let $\Bf = (\B_\Box)_{\Box\in\lambda}$ be an admissible tuple and $(i_0,j_0)\in \lambda$ be an arbitrary box. First, we define for any $\lambda/\mu$
		\begin{align*}
			v_{[\lambda/\mu]} \coloneqq \prod_{\Box\in \lambda/\mu} \widetilde{v}^\lambda_\Box (\Bf) = \B_{(0,0)}^{[(0,0)\in \lambda/\mu]} \prod_{\substack{\Box\in \lambda/\mu, \Box'\not\in \lambda/\mu\\0\leq a,b,c\leq 1\\ \Box'=\Box+(-1)^c (a,b)}} \left(at_1+bt_2+(-1)^c (\B_{\Box'}-\B_{\Box})\right)^{(-1)^{a+b+c}}
		\end{align*}
		where $[P]$ is defined as
		\begin{equation}\label{eqn: [P]}
			[P]\coloneqq
			\begin{cases}
				1, \text{ if } P \text{ is true}\\
				0, \text{ if } P \text{ is false}
			\end{cases}
		\end{equation}
		and we used that any factor of \eqref{eqn: v(Y) defn} occuring in the first product cancels if it involves boxes $\Box,\Box'\in \lambda/\mu$.
		We now want to prove
		\begin{equation}\label{eqn: Idonno}
		 \B_{(i_0,j_0)} \stackrel{!}{=} -i_0 t_1 -j_0 t_2 + \sum_{(i_0,j_0)\in\lambda/\mu} v_{[\lambda/\mu]}.
		\end{equation}
		The claim is trivial if $(i_0,j_0)=(0,0)$ hence we may assume without losing generality that $i_0>0$. Let $\lambda'$ be the partition $\lambda' = \lambda_1\geq \ldots \geq \lambda_{l(\lambda)-1}$ which has degree $\lvert\lambda'\rvert = \lvert \lambda\rvert-\lambda_0<\lvert\lambda\rvert$. We will identify the Young diagram of $\lambda'$ with the set of boxes $(i,j)\in\lambda$ with $i>0$. Using this identification we denote $\Bf' \coloneqq (\B_\Box)_{\Box\in\lambda'}$ and
		\[
			v'_{[\lambda'/\mu']} \coloneqq \prod_{\Box\in\lambda'/\mu'} \widetilde{v}_{\Box}^{\lambda'}(\Bf').
		\]
		In this case one can express the right hand side of \eqref{eqn: Idonno} in the following way:
		\begin{align*}
			&-i_0 t_1 -j_0 t_2 + \sum_{(i_0,j_0)\in\lambda'/\mu'} \sum_{\substack{\lambda/\mu\text{ s.t.}\\(\lambda/\mu)\cap \lambda' = \lambda'/\mu'}} v_{[\lambda/\mu]}\\
			&= -i_0 t_1 -j_0 t_2 + \B_{(0,0)} \\
			&+ \sum_{(i_0,j_0)\in\lambda'/\mu'} v'_{[\lambda'/\mu']}\left(\frac{t_1+ \B_{(1,0)}-\B_{(0,0)}}{\B_{(1,0)}}\right)^{\delta_{\lambda'/\mu',\lambda'}}\prod_{\substack{l=h_{\lambda'/\mu'}\\l>0}}^{\lambda_1-1} \frac{t_1+ \B_{(1,l)}-\B_{(0,l)}}{t_1+t_2+\B_{(1,l)}-\B_{(0,l-1)}}\\
			&+ \sum_{(i_0,j_0)\in\lambda'/\mu'} v'_{[\lambda'/\mu']}\left(\frac{t_1+ \B_{(1,0)}-\B_{(0,0)}}{\B_{(1,0)}}\right)^{\delta_{\lambda'/\mu',\lambda'}}\sum_{\substack{j=h_{\lambda'/\mu'}\\j>0}}^{\lambda_1-1}\frac{t_2+\B_{(0,j)}-\B_{(0,j-1)}}{t_1+t_2+\B_{(1,j)}-\B_{(0,j-1)}}\\&\times\prod_{\substack{l=h_{\lambda'/\mu'}\\l>0}}^{j-1} \frac{t_1+ \B_{(1,l)}-\B_{(0,l)}}{t_1+t_2+\B_{(1,l)}-\B_{(0,l-1)}}\\
			&= -i_0 t_1 -j_0 t_2 + \B_{(0,0)} \\
			&+ \sum_{(i_0,j_0)\in\lambda'/\mu'} v'_{[\lambda'/\mu']}\left(\frac{t_1+ \B_{(1,0)}-\B_{(0,0)}}{\B_{(1,0)}}\right)^{\delta_{\lambda'/\mu',\lambda'}}\prod_{\substack{l=h_{\lambda'/\mu'}\\l>0}}^{\lambda_1-1} \frac{t_1+ \B_{(1,l)}-\B_{(0,l)}}{t_1+t_2+\B_{(1,l)}-\B_{(0,l-1)}}\\
			&+ \sum_{(i_0,j_0)\in\lambda'/\mu'} v'_{[\lambda'/\mu']}\left(\frac{t_1+ \B_{(1,0)}-\B_{(0,0)}}{\B_{(1,0)}}\right)^{\delta_{\lambda'/\mu',\lambda'}}\sum_{\substack{j=h_{\lambda'/\mu'}\\j>0}}^{\lambda_1-1}\left(1-\frac{t_1+\B_{(1,j)}-\B_{(0,j)}}{t_1+t_2+\B_{(1,j)}-\B_{(0,j-1)}}\right)\\ &\times\prod_{\substack{l=h_{\lambda'/\mu'}\\l>0}}^{j-1} \frac{t_1+ \B_{(1,l)}-\B_{(0,l)}}{t_1+t_2+\B_{(1,l)}-\B_{(0,l-1)}}\\
			&= -i_0 t_1 -j_0 t_2 + \B_{(0,0)}+ \sum_{(i_0,j_0)\in\lambda'/\mu'\neq\lambda'} v'_{[\lambda'/\mu']} + t_1+ \B_{(1,0)}-\B_{(0,0)}\\
			&= -(i_0-1) t_1 -j_0t_2 + \sum_{(i_0,j_0)\in\lambda'/\mu'} v'_{[\lambda'/\mu']}
		\end{align*}
		where we set
		\[
		h_{\lambda'/\mu'} \coloneqq \min\Set{j| (1,j)\in \lambda'/\mu'}\in\BN_0\cup\Set{\infty}.
		\]
		 The claim now follows by induction on $\lvert\lambda\rvert$.
	\end{proof}
	We now determine the Jacobian determinant of the above bijection. This comes up in the proof of Theorem \ref{thm: mainthmproof}.
	\begin{lemma}\label{lemma: bethedet}
		For any admissible $\Bf$ as above, the Jacobian matrix of the above map i.e. 
		\[
			M = \left(\frac{\partial \widetilde{v}_\Box^\lambda(\Bf)}{\partial \B_{\Box'}}\right)_{\Box,\Box'\in\lambda}
		\]
		has determinant
		\[
			\prod_{\substack{(i,j)\neq (k,l)\in\lambda\\ 0\leq a,b\leq 1\\ (k,l) = (i+a,j+b)}} (at_1+bt_2+\B_{k,l}-\B_{i,j})^{(-1)^{a+b}}.	
		\]
	\end{lemma}
	\begin{proof}
		We have 
		\[
			\begin{Vmatrix}
				M
			\end{Vmatrix} 
			= \B_{0,0} \begin{Vmatrix}
				\left(\frac{\partial \widetilde{v}^\lambda_\Box(\Bf) / \partial \B_{\Box'} }{\widetilde{v}^\lambda_\Box(\B)}\right)_{\Box,\Box'\in\lambda}
			\end{Vmatrix}
		\]
		and furthermore
		\[
			\frac{\partial \widetilde{v}^\lambda_\Box(\Bf) / \partial \B_{\Box'} }{\widetilde{v}^\lambda_\Box(\Bf)} = \frac{\delta_{\Box,0}\delta_{\Box',0}}{\B_{0,0}} + \sum_{\substack{\Box''\neq\Box \text{ s.t.}\\ \exists 0\leq a,b,c\leq 1:\\ \Box''\coloneqq\Box+(-1)^c (a,b)\in\lambda}}  \frac{(-1)^{a+b}(\delta_{\Box',\Box''}-\delta_{\Box',\Box})} {at_1+bt_2+(-1)^c (\B_{\Box''}-\B_{\Box})}.
		\]
		If one removes the first summand in the above, then the matrix would have determinant 0. Indeed, it is easily checked that $(1,\ldots,1)$ is in the kernel. By looking at the Leibniz formula for the determinant it therefore follows that 
		\[
			\begin{Vmatrix}
				M
			\end{Vmatrix}
			= \begin{Vmatrix}
				\left(\frac{\partial \widetilde{v}^\lambda_\Box(\Bf) / \partial \B_{\Box'} }{\widetilde{v}^\lambda_\Box(\Bf)}\right)_{(0,0)\neq\Box,\Box'\in\lambda}
			\end{Vmatrix}.
		\]
		We now note that the above expression is a minor of the Laplacian of the undirected weighted graph $\Gamma_\lambda$ defined as follows: \\
		Its vertices are given by the boxes in $\lambda$ and two $\Box,\Box'\in\lambda$ are connected by an edge if there are $0\leq a,b,c\leq 1$ with $(a,b)\neq (0,0)$ so that $\Box' = \Box + (-1)^c (a,b)$. In this case the weight of the corresponding edge is given by 
		\[
			w_{\Box,\Box'} = \frac{(-1)^{a+b+1}} {at_1+bt_2+(-1)^c (\B_{\Box'}-\B_{\Box})}.
		\]
		In particular, the whole graph is a union of cycles of length 3 and for each such cycle consisting of the vertices $\Box,\Box',\Box''$ we have
		\begin{equation}\label{eqn: triang weights}
			w_{\Box,\Box'} w_{\Box',\Box''} + w_{\Box',\Box''} w_{\Box'',\Box} + w_{\Box'',\Box} w_{\Box,\Box'} = 0	
		\end{equation}
		or equivalently
		\begin{equation}\label{eqn: triang weights inverse}
			w_{\Box,\Box''}^{-1} + w_{\Box,\Box'}^{-1} + w_{\Box',\Box''}^{-1} = 0.
		\end{equation}
		We will furthermore call cycles consisting of vertices of the shape $(i,j),(i,j+1),(i+1,j+1)\in\lambda$ \textit{upper cycles}. 
		It follows from the weighted matrix tree theorem \cite[Theorem II.3.12]{moderngraphtheory} that
		\[
			\begin{Vmatrix}
				M
			\end{Vmatrix}
			= \sum_{\substack{T\subset \Gamma_{\lambda} \\ \text{spanning tree}}} \prod_{\substack{e = (\Box,\Box')\\ \text{edge in }T}} w_e.
		\]
		We will now define an permutation $\sigma$ (c.f. Figure \ref{fig:weirdTree}) 
		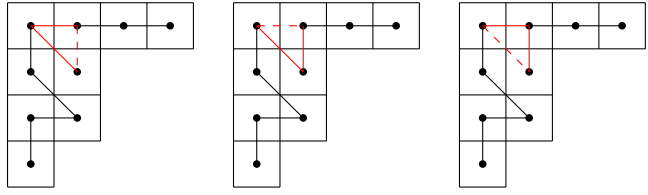
\begin{figure}[h]
			\begin{tikzpicture}[scale = 22/36]
			\filldraw[black] (0,0) circle (2pt) ;
			\filldraw[black] (1,0) circle (2pt);
			\filldraw[black] (2,0) circle (2pt);
			\filldraw[black] (3,0) circle (2pt);
			\filldraw[black] (0,-1) circle (2pt);
			\filldraw[black] (1,-1) circle (2pt);
			\filldraw[black] (0,-2) circle (2pt);
			\filldraw[black] (1,-2) circle (2pt);
			\filldraw[black] (0,-3) circle (2pt);
			\draw (-0.5,0.5) -- (3.5,0.5);
			\draw (-0.5,-0.5) -- (3.5,-0.5);
			\draw (-0.5,-1.5) -- (1.5,-1.5);
			\draw (-0.5,-2.5) -- (1.5,-2.5);
			\draw (-0.5,-3.5) -- (0.5,-3.5);
			\draw (-0.5,0.5) -- (-0.5,-3.5);
			\draw (0.5,0.5) -- (0.5,-3.5);
			\draw (1.5,0.5) -- (1.5,-2.5);
			\draw (2.5,0.5) -- (2.5,-0.5);
			\draw (3.5,0.5) -- (3.5,-0.5);
			\draw (0,0) -- (0,-1);
			\draw (0,-1) -- (1,-2);
			\draw (0,-2) -- (0,-3);
			\draw (0,-2) -- (1,-2);
			\draw (1,0) -- (2,0);
			\draw (2,0) -- (3,0);
			\draw[red] (0,0) --  (1,0) ;
			\draw[red] (0,0) -- (1,-1);
			\draw[red, dashed] (1,0) -- (1,-1);
		\end{tikzpicture}
		\ \ \
		\begin{tikzpicture}[scale = 22/36]
			\filldraw[black] (0,0) circle (2pt) ;
			\filldraw[black] (1,0) circle (2pt);
			\filldraw[black] (2,0) circle (2pt);
			\filldraw[black] (3,0) circle (2pt);
			\filldraw[black] (0,-1) circle (2pt);
			\filldraw[black] (1,-1) circle (2pt);
			\filldraw[black] (0,-2) circle (2pt);
			\filldraw[black] (1,-2) circle (2pt);
			\filldraw[black] (0,-3) circle (2pt);
			\draw (-0.5,0.5) -- (3.5,0.5);
			\draw (-0.5,-0.5) -- (3.5,-0.5);
			\draw (-0.5,-1.5) -- (1.5,-1.5);
			\draw (-0.5,-2.5) -- (1.5,-2.5);
			\draw (-0.5,-3.5) -- (0.5,-3.5);
			\draw (-0.5,0.5) -- (-0.5,-3.5);
			\draw (0.5,0.5) -- (0.5,-3.5);
			\draw (1.5,0.5) -- (1.5,-2.5);
			\draw (2.5,0.5) -- (2.5,-0.5);
			\draw (3.5,0.5) -- (3.5,-0.5);
			\draw (0,0) -- (0,-1);
			\draw (0,-1) -- (1,-2);
			\draw (0,-2) -- (0,-3);
			\draw (0,-2) -- (1,-2);
			\draw (1,0) -- (2,0);
			\draw (2,0) -- (3,0);
			\draw[red,dashed] (0,0) --  (1,0) ;
			\draw[red] (0,0) -- (1,-1);
			\draw[red] (1,0) -- (1,-1);
		\end{tikzpicture}
		\ \ \
		\begin{tikzpicture}[scale = 22/36]
			\filldraw[black] (0,0) circle (2pt) ;
			\filldraw[black] (1,0) circle (2pt);
			\filldraw[black] (2,0) circle (2pt);
			\filldraw[black] (3,0) circle (2pt);
			\filldraw[black] (0,-1) circle (2pt);
			\filldraw[black] (1,-1) circle (2pt);
			\filldraw[black] (0,-2) circle (2pt);
			\filldraw[black] (1,-2) circle (2pt);
			\filldraw[black] (0,-3) circle (2pt);
			\draw (-0.5,0.5) -- (3.5,0.5);
			\draw (-0.5,-0.5) -- (3.5,-0.5);
			\draw (-0.5,-1.5) -- (1.5,-1.5);
			\draw (-0.5,-2.5) -- (1.5,-2.5);
			\draw (-0.5,-3.5) -- (0.5,-3.5);
			\draw (-0.5,0.5) -- (-0.5,-3.5);
			\draw (0.5,0.5) -- (0.5,-3.5);
			\draw (1.5,0.5) -- (1.5,-2.5);
			\draw (2.5,0.5) -- (2.5,-0.5);
			\draw (3.5,0.5) -- (3.5,-0.5);
			\draw (0,0) -- (0,-1);
			\draw (0,-1) -- (1,-2);
			\draw (0,-2) -- (0,-3);
			\draw (0,-2) -- (1,-2);
			\draw (1,0) -- (2,0);
			\draw (2,0) -- (3,0);
			\draw[red] (0,0) --  (1,0) ;
			\draw[red,dashed] (0,0) -- (1,-1);
			\draw[red] (1,0) -- (1,-1);
		\end{tikzpicture}
		\caption{On the left a spanning tree $T$ in $\Gamma_\lambda$ for $\lambda = (4,2,2,1)$, in the middle $\sigma(T)$ and to the right $\sigma^2(T)$. The red edges belong to the unique upper cycle intersecting $T$ along two edges.}
		\label{fig:weirdTree}
		\end{figure}
		on the set of all such spanning trees which will help us remove some of the summands. For this we first fix an ordering on the set of all upper cycles. For a given tree $T$ we then let $C(T)$ be the first upper cycle so that two of its edges are in $T$. If no such cycle exists we set $\sigma(T) \coloneqq T$. Otherwise let $e_1,e_2,e_3$ be the edges of $C(T)$ named in counter-clockwise direction so that $e_1$ is not in $T$ whereas $e_2,e_3$ are in $T$. Removing $e_2$ from $T$ will turn the tree into a forest consisting of two connected components - one containing $e_3$ and the other containing the vertex incident to both $e_1$ and $e_2$. Hence by adding $e_1$ into the subgraph we obtain another spanning tree different from $T$ which we denote by $\sigma(T)$. Since the upper cycles are pairwise edge-disjoint we get $C(\sigma(T)) = C(T)$ if $\sigma(T)\neq T$ and hence one easily sees that $\sigma$ has order $3$. Moreover it follows from \eqref{eqn: triang weights} that those spanning trees for which $\sigma(T)\neq T$ cancel in the sum, which yields
		\[
		\begin{Vmatrix}
			M
		\end{Vmatrix}
		= \sum_{\substack{T\subset \Gamma_{\lambda} \\ \text{spanning tree}\\ \sigma(T) = T}} \prod_{\substack{e = (\Box,\Box')\\ \text{edge in }T}} w_e.
		\]
		We now claim that those trees $T$ with $\sigma(T)=T$ are the same as sets of non-diagonal edges which have exactly one edge in common with every upper cycle and contain each horizontal or vertical edge that is not part of an upper cycle. If we assume this to be true, then it follows from \eqref{eqn: triang weights inverse} that
		\begin{align*}
			&\sum_{\substack{T\subset \Gamma_{\lambda} \\ \text{spanning tree}\\ \sigma(T) = T}} \prod_{\substack{e = (\Box,\Box')\\ \text{edge in }T}} w_e \\
			&= \prod_{(i,j),(i,j+1)\in\lambda} w_{(i,j),(i,j+1)}\prod_{(i,j),(i+1,j)\in\lambda} w_{(i,j),(i+1,j)}\prod_{\substack{C = (e_{hor},e_{vert},e_{diag}) \\ \text{ upper cycle}}} (w_{e_{hor}}^{-1} + w_{e_{vert}}^{-1})\\
			& = \prod_{(i,j),(i,j+1)\in\lambda} w_{(i,j),(i,j+1)}\prod_{(i,j),(i+1,j)\in\lambda} w_{(i,j),(i+1,j)}\prod_{\substack{C = (e_{hor},e_{vert},e_{diag}) \\ \text{ upper cycle}}} (-w_{e_{diag}}^{-1})
		\end{align*}
		which is what we wanted to show.\\
		To show the characterization of trees with $\sigma(T)=T$ we first note that each subset as described above is a spanning tree of $\Gamma_\lambda$. Indeed, there are $\sum_{i=1}^{l(\lambda)-1} (\lambda_i-1)$ many upper cycles and $l(\lambda)-1+\sum_{i=0}^{l(\lambda)-1} (\lambda_i-\lambda_{i+1})$ many edges not part of an upper cyle, hence any set of edges as described above has $\lvert\lambda\rvert-1$ edges and is incident to all boxes, which makes it a spanning tree if it is connected. And indeed one easily sees that each box $(i,j)\in\lambda$ in the subgraph is connected to either $(i+1,j)$ or $(i,j+1)$ and hence each vertex is connected to $(l(\lambda)-1,0)$. \\
		Conversely, the above calculation shows that each spanning tree $T$ with $\sigma(T)=T$ must have exactly one edge in common with any upper cycle and must contain all edges not part of an upper edge. If it contained a diagonal edge $e$, then we could choose it so that there is no other diagonal edge below the diagonal line going through the vertices of $e$. Removing $e$ from $T$ would then make its two adjacent vertices $(i,j)$ and $(i+1,j+1)$ lie in two distinct connected components. However, by the same argument as above, both boxes must be connected to $(l(\lambda)-1,0)$ which is a contradiction. This concludes the proof.
	\end{proof}
	Next, we relate the Jacobian matrices of two kinds of Bethe equations arising during the proof of Theorem \ref{thm: mainthmproof}.
	\begin{lemma}\label{prop: bethematrices}
		For a fixed partition $\lambda$, $\Bf = (\B_{\Box})_{\Box\in\lambda}$ and $\bu = (u_{\lambda/\mu})_{\lambda/\mu}$ free variables we set 
		\[
			M_{\text{Bethe}}(\Bf) \coloneqq \left(\frac{\partial F_{\Box'}(\Bf) / \partial \B_{\Box}}{F_{\Box'}(\Bf)}\right)_{\Box,\Box'\in\lambda}
		\]
		and
		\[
			M_{\text{skew}}(\bu) \coloneqq \left(\frac{\delta_{\lambda/\mu,\lambda/\mu'}}{u_{\lambda/\mu}}+\frac{\partial \overline{F}^\lambda_{\lambda/\mu'}(\overline{\Bf}^\lambda(\bu)) / \partial u_{\lambda/\mu}}{\overline{F}^\lambda_{\lambda/\mu'}(\overline{\Bf}^\lambda(\bu))} \right)_{\lambda/\mu,\lambda'/\mu'}.
		\]
		We now claim that given $\bv = (\bv_\Box)_{\Box\in\lambda}$ and $\bu(\bv) = (u_{\lambda/\mu}(\bv))_{\lambda/\mu}$ with $u_{\lambda/\mu}(\bv)\coloneqq \prod_{\Box\in \lambda/\mu} v_\Box$ we have
		\begin{align}\label{eqn: detclaim}
			\prod_{\lambda/\mu} u_{\lambda/\mu}(\bv)\cdot \begin{Vmatrix} M_\text{skew}(\bu(\bv))\end{Vmatrix} = \prod_{\Box\in\lambda} v_\Box \cdot \begin{Vmatrix} \left(\frac{\partial \widetilde{\B}^\lambda_\Box(\bv)}{\partial v_{\Box'}}\right)_{\Box,\Box'\in\lambda} \end{Vmatrix}\cdot\begin{Vmatrix} M_\text{Bethe}(\widetilde{\Bf}^\lambda(\bv))\end{Vmatrix}
		\end{align}
		and for $Q$ the matrix
		\[
		Q = \left(\left[\Box\in\lambda/\mu\right]\right)_{\lambda/\mu,\Box}
		\]
		we have
		\[
			Q^T\cdot M_\text{skew}(\bu(\bv))^{-1} \cdot Q =  M_\text{Bethe}(\widetilde{\Bf}^\lambda(\bv))^{-1}.
		\]
	\end{lemma}
	\begin{proof}
		From now on we abbreviate $M = M_{\text{skew}}(\bu(\bv))$ and $N = M_{\text{Bethe}}(\widetilde{\Bf}^\lambda (\bv))$. 
		First we show the claim about determinants. 
		Consider the matrix
		\[
			\left( [\Box\geq \Box']\right)_{\Box,\Box'\in\lambda}
		\]
		with $[P]$ as in \eqref{eqn: [P]}.
		Taking a total refinement of the partial ordering on boxes one can realize it as a lower triangular matrix with $1$'s on the diagonal and hence it is invertible of determinant 1. We will write its inverse as $A = (a_{\Box,\Box'})_{\Box,\Box'}$ and extend it to a matrix
		\[
		A' = 
		\begin{blockarray}{ccc}
			& Box & not \ Box \\
			\begin{block}{c(cc)}
				Box & A & 0\\
				not \ Box & 0 & \mathbb{I}\\
			\end{block}
		\end{blockarray}
		\]
		where $\mathbb{I}$ is the identity matrix and \textit{Box} denote the connected skew partitions of $\lambda$ that are of the shape 
		\[
			\overline{\Box} \coloneqq \Set{\Box' | \Box\leq \Box'}
		\]
		for some box $\Box\in\lambda$ and \textit{not Box} the other ones.
		Since we can also write
		\[
			M = \left(\frac{\delta_{\lambda/\mu,\lambda/\mu'}}{u_{\lambda/\mu}}+\sum_{\Box\in \lambda/\mu'}\frac{\partial \widetilde{F}^\lambda_\Box(\overline{\Bf}^\lambda(\bu)) / \partial u_{\lambda/\mu}}{\widetilde{F}^\lambda_\Box(\overline{\Bf}^\lambda(\bu))} \right)_{\lambda/\mu,\lambda/\mu'}
		\]
		this gives 
		\begin{align*}
			&M \cdot A' =\\& \begin{blockarray}{ccc}
				& Box & not \ Box \\
				\begin{block}{c(cc)}
					Box & \left(\frac{a_{\Box,\Box'}}{u_{\overline{\Box}}}+\frac{\partial \widetilde{F}^\lambda_{\Box'}(\overline{\Bf}^\lambda(\bu)) / \partial u_{\overline{\Box}}}{\widetilde{F}^\lambda_{\Box'}(\overline{\Bf}^\lambda(\bu))} \right)_{\Box,\Box'} & \left(\sum_{\Box'\in \lambda/\mu}\frac{\partial \widetilde{F}^\lambda_{\Box'}(\overline{\Bf}^\lambda(\bu)) / \partial u_{\overline{\Box}}}{\widetilde{F}^\lambda_{\Box'}(\overline{\Bf}^\lambda(\bu))}\right)_{\Box,\lambda/\mu}\\
					not \ Box & \left(  \frac{\partial \widetilde{F}^\lambda_\Box(\overline{\Bf}^\lambda(\bu)) / \partial u_{\lambda/\mu}}{\widetilde{F}^\lambda_\Box(\overline{\Bf}^\lambda(\bu))} \right)_{\lambda/\mu,\Box} & \left( \frac{\delta_{\lambda/\mu,\lambda/\mu'}}{u_{\lambda/\mu}}+\sum_{\Box\in \lambda/\mu'}\frac{\partial \widetilde{F}^\lambda_\Box(\overline{\Bf}^\lambda(\bu)) / \partial u_{\lambda/\mu}}{\widetilde{F}^\lambda_\Box(\overline{\Bf}^\lambda(\bu))} \right)_{\lambda/\mu,\lambda/\mu'}\\
				\end{block}
			\end{blockarray}
		\end{align*}
		Using further row operations we can get rid of most terms in the \textit{not Box} rows. More precisely writing
		\[
		B = 
		\begin{blockarray}{ccc}
			& Box & not \ Box \\
			\begin{block}{c(cc)}
				Box & \mathbb{I} & \left(-[\Box\in \lambda/\mu]\right)_{\Box,\lambda/\mu}\\
				not \ Box & 0 & \mathbb{I}\\
			\end{block}
		\end{blockarray}
		\]
		we get
		\[
		M \cdot A' \cdot B = \begin{blockarray}{ccc}
			& Box & not \ Box \\
			\begin{block}{c(cc)}
				Box & \left(\frac{a_{\Box,\Box'}}{u_{\overline{\Box}}}+\frac{\partial \widetilde{F}^\lambda_{\Box'}(\overline{\Bf}^\lambda(\bu)) / \partial u_{\overline{\Box}}}{\widetilde{F}^\lambda_{\Box'}(\overline{\Bf}^\lambda(\bu))} \right)_{\Box,\Box'} &\left(\frac{1}{u_{\overline{\Box}}}\sum_{\Box'\in \lambda/\mu} a_{\Box',\Box}\right)_{\Box,\lambda/\mu} \\
				not \ Box & \left(  \frac{\partial \widetilde{F}^\lambda_\Box(\overline{\Bf}^\lambda(\bu)) / \partial u_{\lambda/\mu}}{\widetilde{F}^\lambda_\Box(\overline{\Bf}^\lambda(\bu))} \right)_{\lambda/\mu,\Box}& \left( \frac{\delta_{\lambda/\mu,\lambda/\mu'}}{u_{\lambda/\mu}} \right)_{\lambda/\mu,\lambda/\mu'}\\
			\end{block}
		\end{blockarray}
		\]
		Now we subsitute $\bu = \bu(\bv)$ into the matrix and using
		\[
			\frac{\partial}{\partial v_\Box} = \sum_{\Box\in \lambda/\mu} \frac{u_{\lambda/\mu}}{v_\Box} \frac{\partial}{\partial u_{\lambda/\mu}}
		\]
		we see that
		\[
		 C\cdot M\cdot A'\cdot B = \begin{blockarray}{ccc}
			& Box & not \ Box \\
			\begin{block}{c(cc)}
				Box & N' & \left( \frac{\partial \widetilde{F}^\lambda_\Box(\overline{\Bf}^\lambda(\bu)) / \partial u_{\lambda/\mu}}{\widetilde{F}^\lambda_\Box(\overline{\Bf}^\lambda(\bu))} \right)_{\Box,\lambda/\mu}\\
				not \ Box & 0 & \left( \frac{\delta_{\lambda/\mu,\lambda/\mu'}}{u_{\lambda/\mu'}} \right)_{\lambda/\mu,\lambda/\mu'}\\
			\end{block}
		\end{blockarray}
		\]
		for
		\[
		C = 
		\begin{blockarray}{ccc}
			& Box & not \ Box \\
			\begin{block}{c(cc)}
				Box & \left([\Box\geq \Box']\frac{u_{\overline{\Box'}}}{v_{\Box}}\right)_{\Box,\Box'} & \left([\Box\in \lambda/\mu]\frac{u_{\lambda/\mu}}{v_{\Box}}\right)_{\Box,\lambda/\mu}\\
				not \ Box & 0 & \mathbb{I}\\
			\end{block}
		\end{blockarray}
		\]
		and 
		\[
		N' = \left(\frac{\delta_{\Box,\Box'}}{v_{\Box}}+\frac{\partial \widetilde{F}^\lambda_{\Box'}(\widetilde{\Bf}^\lambda(\bv)) / \partial v_{\Box}}{\widetilde{F}^\lambda_{\Box'}(\widetilde{\Bf}^\lambda(\bv))} \right)_{\Box,\Box'}
		\]
		Finally note that because of Lemma \ref{lemma: box iso bethe} we have $F_{\Box}(\widetilde{\Bf}^\lambda(\bv)) = v_{\Box}\cdot \widetilde{F}^\lambda_{\Box}(\widetilde{\Bf}^\lambda(\bv))$ and therefore 
		\[
			N' = D \cdot N 
		\]
		with
		\[
			D = \left(\frac{\partial \widetilde{\B}^\lambda_{\Box'} (\bv)}{\partial v_{\Box}}\right)_{\Box,\Box'\in\lambda}.
		\]
		Hence we get
		\begin{align*}
			\detty{M} = \detty{C}^{-1}\detty{CMA'B}
			= 
			\detty{D} \detty{N}
			\prod_{\lambda/\mu} u_{\lambda/\mu}^{-1} \prod_{\Box\in\lambda}v_{\Box}
		\end{align*}
		which proves the first claim.
		For the second claim we need to show
		\[
			Q^T M^{-1} Q = N^{-1}
		\]
		for 
		\[
			Q = \begin{blockarray}{cc}
			& Box \\
			\begin{block}{c(c)}
				Box & \left([\Box\leq \Box']\right)_{\Box,\Box'} \\
				not \ Box & \left([\Box\in \lambda/\mu]\right)_{\lambda/\mu,\Box} \\
			\end{block}
			\end{blockarray}
		\]
		Indeed, it is easy to see that 
		\[
			Q^T A' B= \begin{blockarray}{ccc}
			& Box & not \ Box \\
			\begin{block}{c(cc)}
				Box & \mathbb{I} & 0 \\
			\end{block}
			\end{blockarray}
		\]
		and 
		\[
			C Q= \begin{blockarray}{cc}
			& Box \\
			\begin{block}{c(c)}
				Box & D\\
				not \ Box &\left([\Box\in \lambda/\mu]\right)_{\lambda/\mu,\Box}\\
			\end{block}
		\end{blockarray}
		\]
		which implies
		\[
			Q^T M^{-1} Q = Q^T A'B( CMA'B)^{-1} C Q = N^{-1}
		\]
		as desired.
	\end{proof}
	\printbibliography
\end{document}